\newif\ifpictures
\picturestrue

\documentclass[11pt]{amsart}
\usepackage[T1]{fontenc}
\usepackage{charter}
\usepackage[expert]{mathdesign}
\usepackage{latex_base}
\usepackage{tikz-cd}    
\usepackage{mathrsfs}
\usepackage{mathtools}
\usepackage{macros}
\usepackage{afterpage}
\usepackage{lipsum} 
\usepackage{xcolor}
\usepackage{amsmath}
\usepackage[dash,dot]{dashundergaps}
\usepackage{tikz,ifthen,fp,calc}
\usetikzlibrary{patterns}
\usetikzlibrary{shapes,arrows,shadows,positioning}
\usetikzlibrary{backgrounds}
\tikzset{
  treenode/.style = {
                      align=center},
  root/.style     = {treenode, font=\small},
  env/.style      = {treenode, font=\small},
  dummy/.style    = {circle,draw}
}

\title[]{Counting isolated points outside the image of a polynomial map}
\thanks{Part of this work was supported by the Austrian Science Fund (FWF) P33003}

\author{Boulos El Hilany}

\begin{document}

\maketitle

\begin{abstract} 
We consider a generic family of polynomial maps $f:=(f_1,f_2):\mathbb{C}^2\rightarrow\mathbb{C}^2$ with given supports of polynomials, and degree $\deg f:=\max (\deg f_1, \deg f_2)$. We show that the (non-) properness of maps $f$ in this family depends uniquely on the pair of supports and that the set of isolated points in $\mathbb{C}^2\setminus f(\mathbb{C}^2)$ has a size of at most $6\deg f$. This improves an existing upper bound $(\deg f - 1)^2$ proven by Jelonek. Moreover, for each $n\in\mathbb{N}$, we construct a dominant map $f$ above, with $\deg f = 2n+2$, and having $2n$ isolated points in $\mathbb{C}^2\setminus f(\mathbb{C}^2)$. 

Our proofs are constructive and can be adapted to a method for computing isolated missing points of $f$. As a byproduct, we describe those points in terms of singularities of the bifurcation set of $f$.

\end{abstract}

 \markleft{}
 \markright{}
\section{Introduction}\label{sec:intro} A polynomial map $f:=(f_1,\ldots,f_n):\C^n\to\C^n$ is said to be \emph{dominant} if it sends $\C^n$ onto a dense subset of $\C^n$. Little is known about the set $\C^n\setminus f(\C^n)$ of \emph{missing points} for a dominant map $f$. The scarcity of tools for describing the topology of a map leaves several good questions unanswered. Notably, those concerning the geometrical aspects of the set of missing points. For example, how many equi-dimensional components can this set have? What kind of singularities can appear?

The set of missing points measures how far a map is from being surjective. Characterizing such topological properties is a classical problem with numerous applications, especially  since polynomial maps appear in algebraic statistics~\cite{DSS09}, computer vision~\cite{ZiRi03}, robotics~\cite{Lam70}, optimization~\cite{Las15} and other areas. 

A polynomial map usually constitutes a model for above applications in which the input data is a point $p$ in $\C^n$, and the solution to the problem is the output $f^{-1}(p)$. Consequently, one avoids the vicinity of the set of missing points in order to guarantee a manageable output. This brings forward the following problem: \emph{What would be an effective method to compute the set of missing points and to find interesting examples?}

This work is a further step (c.f.~\cite{EH19anote}) towards understanding those maps whose topology possesses an extremal structure. The particular motivation here lies in the two-dimensional setup where the set of missing points consists of zero- and one-dimensional components: We are interested in counting the isolated points in $\C^2\setminus f(\C^2)$. The proofs in this paper are constructive, can be generalized and aimed towards the computational aspects of the set of missing points.

\subsection{Main results}\label{subs:int1} The preimages under $f$ around generic points in $\C^n$ form a locally analytical covering of constant degree $\mu_f$. In other words, $\mu_f$ is the number of preimages of a map over a generic point in $\C^n$. Such is not the case around the set of missing points of $f$. These are contained in the set $\mathcal{J}_f$ of points in $\C^n$ at which $f$ is \emph{non-proper}  (see Definition~\ref{def:Jel}). Jelonek showed~\cite{Jel93} that $\mathcal{J}_f$ (also referred to as the \emph{Jelonek set}) is an algebraic hypersurface ruled by regular rational curves, for which he provided upper bounds on the degree. 

Above findings, together with other results~\cite{Jel99}, play an important role in describing the set of missing points. In fact, Jelonek later showed in~\cite{Jel99a} that for planar maps $f:=(f_1,f_2):\C^2\rightarrow\C^2$, the set $\Em$ of isolated points in $\C^2\setminus f(\C^2)$ is either empty, or has a size of at most $\deg f_1\deg f_2-\mu_f-1$. Moreover, if $\deg f:=\max(\deg f_1,\deg f_2)$, then $|\Em|\leq (\deg f - 1)^2$. Our main result is the improvement (from quadratic to linear) of these upper bounds for a large class of maps.

The \emph{support of a polynomial} is the set of exponent vectors, in the lattice of integers, of monomial terms appearing in the polynomial with non-zero coefficients. The supports of polynomials in a planar map $f$ form a pair of finite subsets in $\N^2$, also called the \emph{support of $f$}. We use the polyhedral properties of the latter to study the geometry of the Jelonek set and of $\C^2\setminus f(\C^2)$. In order to relate these two notions, a genericity assumption on $f$ must be assumed.

 Consider the three curves in $\C^2$, $\{f_1 = 0\}$, $\{f_2 = 0\}$ and the set $C(f)$ of critical points of $f$. We say that $f$ is \emph{generically non-proper} (see Definition~\ref{def:gen-prop}) if the number of isolated points in the intersection of any of these two curves is maximal with respect to the support of $f$. This means that no other map with the same support has more of such isolated points.

\begin{theorem}\label{th:main}
Let $f:\C^2\rightarrow\C^2$ be a generically non-proper map. Then, the number of isolated points in $\C^2\setminus f(\C^2)$ cannot exceed any of the two values $6\deg f$ and
\begin{equation}\label{eq:bound}
\frac{ \deg f_1\deg f_2}{\mu_f(\mu_f-1)} + 2(\deg f_1 + \deg f_2).
\end{equation} 

\end{theorem} We will emphasize the importance of this Theorem from two standpoints: Our upper bounds hold true for very large classes of polynomial maps (Corollary~\ref{cor:1}), and are close to being optimal (Theorem~\ref{th:optimal}).

We use $\C[A]$ to denote the space of all polynomial maps $\C^2\to\C^2$ having a support contained in a pair $A$ of finite subsets of $\N^2$. This means that coordinates of any point $\C[A]$ give a list of coefficients that identifies a planar map (see Example~\ref{eq:map:ex1}).

\begin{theorem}\label{th:family}
Let $A$ be a pair of finite subsets of $\N^2$. Then, all maps in $\C[A]$ that are either proper, or generically non-proper, form a dense subset.
\end{theorem}

Since proper maps have no missing points, we obtain the following immediate consequence.

\begin{corollary}\label{cor:1}
Let $A$ be as in Theorem~\ref{th:family}. Then, there exists a dense subset of maps in $\C[A]$ for which the bounds in Theorem~\ref{th:main} hold true.
\end{corollary}

Our next result how far bounds presented in Theorem~\ref{th:main} are from being sharp.

\begin{theorem}\label{th:optimal}
There exists a pair $A$ of finite subsets in $\N^2$, such that all generically non-proper polynomial maps in $\C[A]$ have degree $2n+2$, and have $2n$ missing isolated points. 
\end{theorem}

With generic enough coefficients, a well-chosen collection of supports can be sufficient to construct a map with an interesting topology~\cite{EH19anote}. Theorem~\ref{th:optimal} is another instance of this approach.

All the main results in this paper are obtained by carefully analyzing the pair of supports of a polynomial map. As a byproduct, we discover a universality result (See Lemmas~\ref{lem:mu=1} and~\ref{lem:mu>1}).

\begin{theorem}\label{th:univers}
Let $k$ be a positive integer. Then, there exists a polynomial map $f:\C^2\to\C^2$ such that $\mu_f=k+1$ and $\Em\neq\emptyset$ if and only if $k\geq 1$.
\end{theorem}

Another byproduct of our study is a set of examples suggesting that the missing points of a map can be described uniquely from the geometry of the Jelonek set and the discriminant of a map (that is, the set $f(C(f))$). \textit{The below map $\C^2_{u,v}\to\C^2_{s,t}$, is taken from the proof of Lemma~\ref{lem:mu>1} for $k=2$, has $(0,1)$ as the only missing point, and its Jelonek set has a singularity at $(0,1)$ formed by a horizontal line and a node of the curve $\{1 - 2t + t^2 -s^2 + s^3 = 0\}$:}
\begin{equation}\label{eq:map:ex1}
  \begin{array}{@{}ccccl@{}}
   (u,v) & \mapsto & (1 - u^2v^2 + u^2v^3,~1+ uv - u^3v^3+2u^3v^4).
  \end{array}
\end{equation}
Missing isolated points can be degenerate intersections among components of the Jelonek set and the discriminant: \textit{The map}  
\begin{equation}\label{eq:map:ex2}
  \begin{array}{@{}ccccl@{}}
   (u,~v) & \mapsto & ((1 - uv)^2,~1+ v +  uv(1 - uv)^2)
  \end{array}
\end{equation} \textit{has $(0,1)$ as the only missing point. Its Jelonek set is} $\{1 +2s - 2t + t^2 - s^2 -2st + s^3 = 0\}$, \textit{having a cusp at $(0,1)$, that belongs to the discriminant $\{s=0\}$.}

Such degenerate intersections/singularities can be milder, and still correspond to missing isolated points: \textit{The map}  
\begin{equation}\label{eq:map:ex3}
  \begin{array}{@{}ccccl@{}}
   (u,~v) & \mapsto & (uv,~v^2-uv^3+ 2v - uv^2 + uv)
  \end{array}
\end{equation} \textit{from Section~\ref{sec:proof-th} has $\{(1,1),(2,2)\}$ as the set of missing isolated points. Its Jelonek set $\{s - t = 0\}\cup \{s - 1 = 0\}$ has a singularity $(1,1)$, and intersects tangentially the subset} $\{-4 -4t +8s - 5s^2 +4st = 0\}$ \textit{of the discriminant at $(2,2)$}.

Some of those phenomena will be explained in Sections~\ref{sec:proof-prop1} and~\ref{sec:proof-prop3}. Others will be thoroughly examined and generalized in a future work.

\subsection{On attaining the upper bounds}\label{subs:int2} 
One deduces from the definitions that the set $\Em$ of isolated missing points is contained in the Jelonek set $\mathcal{J}_f$ of $f:\C^2\to\C^2$. The straightforward approach, considered by Jelonek in~\cite{Jel99a} to approximate $\Em$ is to study $\mathcal{J}_f$.

Although we follow a similar strategy, our methods are more refined than those in~\cite{Jel99a}; we take into account the combinatorics of the support $A$ of $f$.

A more general form of our approach has proven to be powerful in describing \emph{mixed discriminants}~\cite{GKZ94,Est10,CCDDS13}, resultants~\cite{St94} and the \emph{bifurcation set} $B_F$ of a polynomial map $F:\C^n\to\C^p$~\cite{NZ90,Zah96,Est13}. This is the complement to the maximal open set $S\subset\C^p$, such that the restriction of $F$ to the preimage $F^{-1}(S)$ is a locally trivial fibration. 

In our setting, the Jelonek set forms half of the bifurcation set $B_f$, for it describes the atypical behavior of $f$ at infinity. This interpretation allows one to identify a generic $p\in\mathcal{J}_f$ with a pair of subsets in the boundary of $A$, called \emph{faces}\cite{EH19desc,Est13} (see also Section~\ref{subs:faces}). Points in $\Em$ are not generic in $\mathcal{J}_f$, hence computing them requires a thorough examination of $f$ at the faces of $A$.

Our analysis is based on the following observation in Section~\ref{subs:proof-S0}: The size of $\Em$ is proportional to the \emph{integer length} (Section~\ref{subsub:integ-leng}) of some faces of $A$. This in turn is proportional to $\deg f$ (Lemma~\ref{lem:degree-length}), and to the topological degree $\mu_f$ (Lemma~\ref{lem:gen-non-degree}). However, the latter quantity is disproportional to the size of $\Em$.

This delicate relation among $\deg f$, $\mu_f$ and $|\Em|$ is summarized in Theorem~\ref{th:main}. Only a small class of supports $A$ have long faces and support maps with small topological degrees. These were used to prove Theorem~\ref{th:optimal}. 

After concluding this section with future directions, we organize the rest of the paper as follows. 

In Section~\ref{sec:proof-th} we prove Theorems~\ref{th:optimal} and~\ref{th:univers}. We also divide Theorem~\ref{th:main} into three Propositions. These will be proven in Sections~\ref{sec:proof-prop1} --~\ref{sec:proof-prop3}.

In Section~\ref{sec:prelim-results}, we introduce the main notations and preliminary results (Proposition~\ref{prop:main}). We define \emph{relevant faces} in Section~\ref{sec:prelim-results} (Definition~\ref{def:face-non-prop}). We show that these play an important role in computing the Jelonek set and in detecting isolated missing points for generically non-proper maps. We prove Theorem~\ref{th:family} at the end of that Section.

\subsection{Future directions}\label{subs:int3} Inequality~\eqref{eq:bound} can be sharpened by refining the analysis made in Sections~\ref{sec:proof-prop1} - ~\ref{sec:proof-prop3}. Such improvement is not included here since it makes the proofs too cumbersome and would not add significant content to the goals of this work.

Regarding the validity of upper bounds in Theorem~\ref{th:main} for maps that are \emph{not} generically non-proper, the problem remains open. However, the analysis made here is beneficial to identify potential candidates for maps $f:\C^2\to\C^2$ whose size of $\Em$ is proportional to $\deg^2 f$. This direction is the subject of our future work.

Methods presented here can be generalized for higher dimensions and for rational maps. This is because one can still describe the Jelonek set in these settings using faces of tuples of polytopes~\cite{EH19desc,Est13}.

Another equally important future aim is to tackle this problem for the case where the map and the spaces are real. This is more pertinent to applications mentioned in the beginning. Relevant to this direction are ample works of Fernando, Gamboa and Ueno~\cite{Fer14,Fer16,Fer03,FerGam11,FerUeno14} on describing the possible images that $\R^n$ can take under a polynomial map. In fact, Fernando showed that the complement of any finite set $S$ in $\R^n$ is the image of a polynomial map $\R^m\to\R^n$ for some $m\geq n$~\cite{Fer03}. The polynomial map constructed for proving this highly non-trivial result restricts to a map similar to the one for Theorem~\ref{th:optimal}. Results in the same vain exist for when the above set $S$ is a convex polyhedron~\cite{FerGam11}, or the complement of one~\cite{FerUeno14}.

Concerning the computation of the Jelonek set for real maps, Jelonek described its geometry~\cite{Jel02}, and gave a method for computing it for some classes of maps $\R^2\to\R^2$~\cite{Jel01c}. Stasica-Valette provided a different technique for computing the Jelonek set of the latter planar maps that have finite fibers~\cite{Sta07}.
%
%
%
%
%
%
%
%
%

\subsection*{Acknowledgements} The author is grateful to Zbigniew Jelonek for introducing him to the problem and for his helpful remarks. The author would also like to thank Elias Tsigaridas for fruitful discussions. The author thanks the anonymous referee for their valuable suggestions on a earlier versions of the manuscript, for pointing out mistakes therein and mentioning works of Fernando, Gamboa and Ueno on the subject. The author is grateful to the Mathematical Institute of the Polish Academy of Sciences in Warsaw for their financial support and hospitality.

\section{Proof of Theorems~\ref{th:main},~\ref{th:optimal} and~\ref{th:univers}}\label{sec:proof-th} Consider any pair of subsets $A:=(A_1,A_2)$ in $\C^2$. Bernstein~\cite{Ber75} has shown that there exists a positive integer $V(A)$ (Section~\ref{subs:mixed}) such that any pair of polynomials supported on $A$ cannot have more than $V(A)$ isolated solutions in $(\C^*)^2$. Moreover, equality holds if those polynomials are generic. 

Consider a dominant polynomial map $f:=(f_1,f_2):\C^2\rightarrow\C^2$ having $A$ as support (Section~\ref{subs:supp-roots}). Let $C(f)$ denote the set of points in $\C^2$ at which the Jacobian matrix of $f$ is singular, and let $\Sigma$ denote the support of the determinant $|\Jac f|$ of this Jacobian.

\begin{definition}\label{def:gen-prop}
We say that $f$ is \emph{generically non-proper} if it is dominant, non-proper, satisfies $f(0,0)\in\TT$ and systems
\begin{equation}\label{eq:def:gen}
  \begin{array}{@{}ccccc@{}}
   f_1 & = & f_2 & = & 0,\\
   f_1 & = & |\Jac f| & = & 0,\\
   f_2 & = & |\Jac f| & = & 0,\\
  \end{array}
\end{equation} have respectively $V(A)$, $V(A_1,\Sigma)$ and $V(A_2,\Sigma)$ isolated solutions in $(\C^*)^2$.
\end{definition}

\subsection{Proof of Theorem~\ref{th:main}}\label{subsec:th11} We start with a definition~\cite{Jel93,Jel99}.

\begin{definition}\label{def:Jel}
  Given two affine varieties, $X$ and $Y$, and a map $F:X\to Y$, we say that $F$
  is \emph{non-proper} at a point $y\in Y$, if there is no neighborhood
  $U\subset Y$ of $y$, such that the preimage $F^{-1}(\overline{U})$ is compact,
  where $\overline{U}$ is the Euclidean closure of $U$. In other words, $F$ is
  non-proper at $y$ if there is a sequence of points $\{x_k\}$ in $X$ such that
  $\Vert x_k\Vert\to + \infty$ and $f(x_k)\to y$. The \emph{Jelonek set} of
  $F$, $\mathcal{J}_F$, consists of all points $y\in Y$ at which $F$ is
  non-proper.
\end{definition}
Let $\mathcal{K}_f\subset\mathcal{J}_f$ denote the Jelonek set of the restricted map
\[
  \begin{array}{@{}ccccc@{}}
   f_{|C(f)}: & C(f) & \to & f(C(f)).
  \end{array}
\]
 For example, $\mathcal{K}_f$ coincides with $\{(0,1)\}$ in the map~\eqref{eq:map:ex2} and with $\{(2,2)\}$ for the map in~\eqref{eq:map:ex3}.
 
 We prove the following result in Section~\ref{sec:proof-prop3}. Recall that $\Em$ denotes the set of isolated points in $\C^2\setminus f(\C^2)$.

\begin{proposition}\label{prop:double-points}
Let $f:\C^2\to\C^2$ be a generically non-proper map. Then, we have 
\begin{align*}
|\Em\cap \mathcal{K}_f|& \leq \deg f_1 + \deg f_2.
\end{align*}
\end{proposition} Let $\mathcal{C}^+_f$ be cross in $\C^2$ centered at $f(0,0)$. The next result will be proven in Section~\ref{sec:proof-prop2}. 

%

\begin{proposition}\label{prop:S1f}
Let $f:\C^2\to\C^2$ be a generically non-proper map. Then, we have
\begin{align*}
  |\Em\cap\mathcal{C}_f^+\setminus\mathcal{K}_f |& \leq  \deg f_1 + \deg f_2.
\end{align*}
\end{proposition} Write $\mathcal{C}_f^0$ for the complement of $\mathcal{C}_f^+$ in $\C^2$. The following result will be proven in Section~\ref{sec:proof-prop1}. 

\begin{proposition}\label{prop:S0f}
Let $f:\C^2\to\C^2$ be a generically non-proper map. Then, we have 
\begin{equation}\label{eq:ineqS01}
 |\Em \cap\mathcal{C}_f^\circ\setminus \mathcal{K}_f|  \leq \frac{3\deg f_1\cdot \deg f_2}{4\mu_f(\mu_f - 1)},
\end{equation}
\begin{equation}\label{eq:ineqS02}
 |\Em \cap\mathcal{C}_f^\circ\setminus \mathcal{K}_f| \leq 2\max(\deg f_1,\deg f_2).
\end{equation}
\end{proposition} The bound in Theorem~\ref{th:main} is obtained by observing that $\Em$ can be written as the disjoint union of the three sets appearing in the above three propositions. 

\subsection{Proof of Theorem~\ref{th:optimal}} Let $P,Q$ be two univariate polynomials such that: $\deg P = \deg Q = n$, $\gcd(P,Q)=1$ and $P(0)\cdot Q(0)\neq 0$. Define $f$ to be the map 
\[
  \begin{array}{@{}ccl@{}}
  (u,~v) & \mapsto & (uv,~v^2\cdot P(uv) + v\cdot Q(uv) + uv).
  \end{array}
\]
We will show that all maps defined this way form the family proving Theorem~\ref{th:optimal}. 

Let $y\in(\C^*)^2$ and let $(u,v)\in\C^2$ be a solution to $f-y =0$, that is,
\begin{align*} 
  f_1(u,v) -y_1 &= 0 \\ 
  f_2(u,v) - y_2&=  0.
\end{align*} Then, we have $uv = y_1\neq 0$ and 
\begin{equation}\label{eq:pol:red}
v^2P(y_1) + vQ(y_1) + y_1 - y_2 = 0.
\end{equation} If $P(y_1) = 0$ (resp. $Q(y_1) = 0$), then $Q(y_1) \neq 0$ (resp. $P(y_1) \neq 0$). Therefore, $PQ = 0\Rightarrow y_1\neq y_2$ since otherwise $v=0$, a contradiction to $uv = y_1\neq 0$. This shows that for any $y\in (\C^*)^2$, such that $PQ=0$, we have $f^{-1}(y)=\emptyset \Leftrightarrow y_1 = y_2$. Otherwise, if $P Q\neq 0$, then~\eqref{eq:pol:red} has a non-zero solution $v$ and $u=v/y_1\neq 0$. This shows that $y\in(\C^*)^2$, $y_1\neq y_2\Rightarrow f^{-1}(y)\neq \emptyset$.

We conclude from this analysis that  
\[
  \begin{array}{@{}ccccc@{}}
 (\C^*)^2\setminus f(\C^2)& = & \bigcup_{P(a)\cdot Q(a) = 0}\{(a,a)\}.
  \end{array}
\] The size of this set is equal to $\deg P +\deg Q = 2n$. Furthermore, one can easily check that $f^{-1}(y)\neq \emptyset $ for any $y\in\C^2\setminus \TT$.

Up to now, we have constructed a map having degree $2n+2$ and $\#\Em = 2n$. Obviously, all polynomials $P,Q$ above form a dense family of pairs of univariate polynomials in $\C_n[uv]^2$. This shows that the set of polynomial maps $f$ constructed this way forms a dense subset in $\C[A]$. Finally, since $f$ is non-proper, Theorem~\ref{th:family} yields the proof.

\subsection{Proof of Theorem~\ref{th:univers}} We split the proof into two statements. 
Recall that $\mu_f$ is the number of preimages of a map over a generic point in $\C^2$

\begin{lemma}\label{lem:mu=1}
Let $f:\C^2\to\C^2$ be a dominant map such that $\mu_f = 1$. Then, we have $\Em = \emptyset$.
\end{lemma}

\begin{proof}
For any $y\in \mathcal{J}_f$, the set $f^{-1}(y)$ is either empty or has positive dimension. Recall that $\mathcal{J}_f$ is a curve in $\C^2$~\cite{Jel93}. Then, it is enough to show that the preimage under $f$ has a positive dimension over finitely-many points.

If the preimage of some point $y$ has positive dimension, then it forms a union of distinct irreducible components. Therefore, the set $f^{-1}(y)$ shares a component $C$ with the set $C(f)$ of critical points of $f$.

This set $C(f)$ has finitely-many components $C$ arising this way. Each of them is mapped to a point under $f$.
\end{proof} 
\begin{lemma}\label{lem:mu>1}
Let $k$ be a positive integer. Then, there exists a polynomial map $f:\C^2\to\C^2$ such that $\mu_f=k+1$ and $\Em\neq \emptyset$.
\end{lemma}

\begin{proof}
Let $P$ be a univariate complex polynomial of degree $k$ such that $P(0)\neq 0$. We show that
\[
  \begin{array}{@{}ccccc@{}}
 (u,v)& \mapsto & \big( P(uv) + u^kv^{k+1},~1+uvP(uv)+2u^{k+1}v^{k+2}\big).
  \end{array}
\] satisfies the claims of the Lemma.

To compute $\mu_f$, we fix an arbitrary $y\in\C^2$ and solve $f-y =0$ for $u,v$. Make the formal substitution $s=uv$, $t = v$, and eliminate $t$, we obtain
\begin{equation}\label{eq:mu>1}
s(P(s) - 2y_1) +y_2 - 1 = 0.
\end{equation}
There are $k+1$ distinct values $s\in\C^*$ satisfying~\eqref{eq:mu>1} for a generic choice of $P$ and $y$. This makes for $k+1$ distinct couples $(uv,v)\in(\C^*)^2$. Therefore, we have $\mu_f = k+1$.

One can check that if $P(0)\neq 0$, then $f-y=0$ has no solutions only when $y=(0,1)$.
\end{proof} 

\section{Faces of finite sets and restricted systems}\label{sec:prelim-results} 

In Section~\ref{subs:faces}, we consider pairs of finite sets in a two-dimensional lattice. We introduce the notions of a face, independence, Minkowski sum and mixed volumes. In Section~\ref{subs:supp-roots}, we recall Bernstein's results in~\cite{Ber75}, relating those combinatorial notions to the number of solutions to polynomial equations.


In Section~\ref{subs:toric}, we introduce a monomial change of variables that depends on the faces of pairs of supports. We then illustrate its usefulness in keeping track of solutions outside the complex torus to a polynomial system of equations (Proposition~\ref{prop:main}). Such a description is important in computing isolated missing points of dominant polynomial maps, and will be used in the steps that follows in proving the main upper bounds.

\subsection{Pairs of finite sets, faces and mixed volume}\label{subs:faces} A \emph{polytope} $\Delta$ in $\R^2$ is a bounded intersection of closed half-planes of the form $\{\alpha_0 + \alpha_1X_1 + \alpha_2X_2\geq 0\}$, for some $\alpha_0,\alpha_1,\alpha_2\in\R$. These are called the \emph{supporting half-spaces} of $\Delta$, and their boundary intersects $\partial\Delta$ at a connected set of $\Delta$ called \emph{face}. 

Any face $F$ of $\Delta$ minimizes a function $\alpha^*:\Delta\rightarrow\R$, $(X_1,X_2)\mapsto\alpha_0 + \alpha_1X_1 + \alpha_2X_2$. In this case, we say that $\alpha=(\alpha_1,\alpha_2)$ \emph{supports} $F$. 

We formulate this terminology in terms of finite sets in $\Z^2$.

\subsubsection{Faces of finite sets}\label{subsubs:finite-sets} A \emph{face $\phi$ of a finite subset} $\Sigma\subset\Z^2$ is the intersection of $\Sigma$ with a face of its convex hull. The \emph{supporting vector} of $\phi\subset\Sigma$ is the supporting vector of the convex hull of $\phi$.

The following notation is taken from~\cite[Section 2.2]{Est13}. Recall that the \emph{Minkowski sum} $A+B$ of two subsets $A,B\subset\R^n$ is the set $\{a+b~|~a\in A,b\in B\}$. In Figure~\ref{fig:tuples+basis}, the Minkowski sum of the two subsets on the left is the set on the right.

A \emph{pair of finite sets} in $\Z^2$ refers to a pair $A:=(A_1,A_2)$, in which each \emph{member} $A_i$ $(i=1,2)$ is a finite set of $\Z^2$. A \emph{face} $\Gamma$ of a pair $A$, denoted by $\Gamma\prec A$, is a pair $\Gamma:=(\Gamma_1,\Gamma_2)$ such that $\Gamma_i$ is a face of $A_i$, $i=1,2$ and $\Gamma_1 + \Gamma_2$ is a face of $A_1+A_2$. The pair of sets on the left of Figure~\ref{fig:tuples+basis} have the pair $(\gamma,\delta)$ as a face.

The \emph{dimension} of a face, denoted $\dim \Gamma$, is the dimension of the convex hull of $\Gamma_1 +\Gamma_2$. In Figure~\ref{fig:tuples+basis}, both $(\gamma,\delta)$ and $(a,\blacksquare)$ have dimension one, $(\bigstar,\blacksquare)$ has dimension zero. 

Since a supporting vector $\alpha\in\Q^2$ of $\Gamma_1+\Gamma_2$ is also one for the face $\Gamma_i\subset A_i$ ($i=1,2$) (the function $\alpha^*$ is linear), we say that $\alpha$ \emph{supports} $\Gamma\prec A$ . In Figure~\ref{fig:tuples+basis}, $(-2,1)$ supports $(\gamma,\delta)$ and $(3,-1)$ supports $(a,\blacksquare)$.

We use the convention that a face of a pair is the pair itself if and only if the vector $(0,0)$ supports it. We refere to such face as trivial.


\begin{figure}[h]
\centering
\begin{tikzpicture}
    \tikzstyle{bluefill1} = [fill=blue!20,fill opacity=0.8]          
    \tikzstyle{blackfill1} = [fill=gray,fill opacity=0.4]          
    \tikzstyle{greyfill1} = [fill=gray!20,fill opacity=0.8]          
    \tikzstyle{conefill} = [pattern = north east lines, pattern color=gray]          
    \tikzstyle{ann} = [fill=white,font=\footnotesize,inner sep=1pt] 
    \tikzstyle{ghostfill} = [fill=white]	
    \tikzstyle{ghostdraw} = [draw=black!50]					
	\tikzstyle{ann1} = [font=\footnotesize,inner sep=1pt] 
	\tikzstyle{ann2} = [font=\normalsize,inner sep=1pt] 


\begin{scope}[xshift = -5cm, yshift = 0cm]

\begin{scope}[xshift = 1cm, yshift = 3.2cm, xscale=1.5, yscale=1.5]	   
	\draw[arrows=->, line width=0.5 mm]	(0,0)--(-0.8,0.4); 
	\draw[arrows=->, line width=0.5 mm]	(0,0)--(0.9,-0.3);   
	
	\draw [fill] (0,0) circle [radius=0.03];

	\node[ann2] at (-1.2,0.5)	{{\small $(-2,1)$}};	
	\node[ann2] at (1.4,-0.4)	{{\small $(3,-1)$}};

\end{scope}


\begin{scope}[xshift = -3cm, yshift = 0cm, xscale=1.2, yscale=1.2]


	\filldraw[bluefill1,line width=0.0 mm ](0,0)--(1,2)--(1,2.5)--(0.5,1.5)-- (0,0);	


	\draw[arrows=->,line width=0.3 mm, dotted] (0,0)-- (0,1.3); 
	
	\draw[arrows=->,line width=0.3 mm, dotted] (0,0)-- (1.3,0); 

%
%
%
%
%
	


		\node[ann1] at (0,0)   {$\bigstar$};	
		
		\draw [fill, color = blue] (0.5,1) circle [radius=0.04];

		\draw [fill, color = blue] (1,2) circle [radius=0.04];
		

		\node[ann1] at (1,2.5)   {$\blacktriangle$};	
		
		\draw [fill, color = blue] (0.5,1.5) circle [radius=0.04];

	\node[ann2] at (0.1,1)		{$a$};
	\node[ann2] at (0.4,2)		{$b$};		
	\node[ann2] at (0.9,1)		{$\gamma$};

\begin{scope}[xshift = 2 cm, yshift = 0cm]	


	\filldraw[blackfill1,line width=0.0 mm ](0,0)--(0.5,0)--(1,1)-- (0,0);	


	\draw[arrows=->,line width=0.3 mm, dotted] (0,0)-- (0,1.3); 
	
	\draw[arrows=->,line width=0.3 mm, dotted] (0,0)-- (1.3,0); 
	

%
%
%
	


	\node[ann1] at (0,0)   {$\blacksquare$};	
		

	\node[ann1] at (1,1)   {$\blacklozenge$};	

	\draw [fill] (0.5,0) circle [radius=0.04];
	
	\draw [fill] (0.5,0.5) circle [radius=0.04];
	
	\node[ann2] at (0.35,0.7)		{$c$};
	\node[ann2] at (0.35,-0.3)		{$d$};
	\node[ann2] at (1,0.5)	{$\delta$};

	\end{scope}

 \end{scope}

\end{scope}
	

\begin{scope}[xshift = 0cm, yshift = 0.5cm, xscale=1.2, yscale=1.2]


	\filldraw[greyfill1,line width=0.0 mm ](0,0)--(0.5,0)--(2,3)
										-- (2,3.5) -- (1,2.5) -- (0.5,1.5)-- (0,0);	
										

%
%
%
%
%
%



	\draw [fill, color = gray ] (0.5,0.5) circle [radius=0.04];
	
	\draw [fill, color = black ] (0.5,0) circle [radius=0.06];
	
	\draw [fill, color = black ] (2,3) circle [radius=0.06];
	
	\draw [fill, color = black ] (1,1) circle [radius=0.06];
	
	\draw [fill, color = black ] (1.5,2) circle [radius=0.06];

	\draw [fill, color = gray ] (2,3.5) circle [radius=0.04];

	\draw [fill, color = gray] (0,0) circle [radius=0.04];	
	
	\draw [fill, color = gray] (0.5,1) circle [radius=0.04];

	\draw [fill, color = gray] (1,2) circle [radius=0.04];
		
	\draw [fill, color = gray] (1,2.5) circle [radius=0.04];

	\draw [fill, color = gray] (0.5,1.5) circle [radius=0.04];

	\draw [fill, color = gray] (0.5,0.5) circle [radius=0.04];	
	
	\draw [fill, color = gray] (1,1.5) circle [radius=0.04];

	\draw [fill, color = gray] (1.5,2.5) circle [radius=0.04];
		
	\draw [fill, color = gray] (1.5,3) circle [radius=0.04];




	\draw[arrows=->,line width=0.3 mm, dotted] (0,-0.5)-- (0,1.3); 
	
	\draw[arrows=->,line width=0.3 mm, dotted] (-0.5,0)-- (1.3,0);


	\draw[arrows=->, line width=0.5 mm, color = red]	(0,0)--(0.5,1); 
	\draw[arrows=->, line width=0.5 mm, color = red]	(0,0)--(-0.5,-0.5);   
	

	\node[ann2] at (1.9,1.4)	{$\gamma +\delta$};
	

\end{scope}	
\end{tikzpicture}
\caption{\textbf{(L)}: Pair $(A_1,A_2)$ of finite subsets. \textbf{(R)}: Their Minkowski sum.}\label{fig:tuples+basis}
\end{figure}
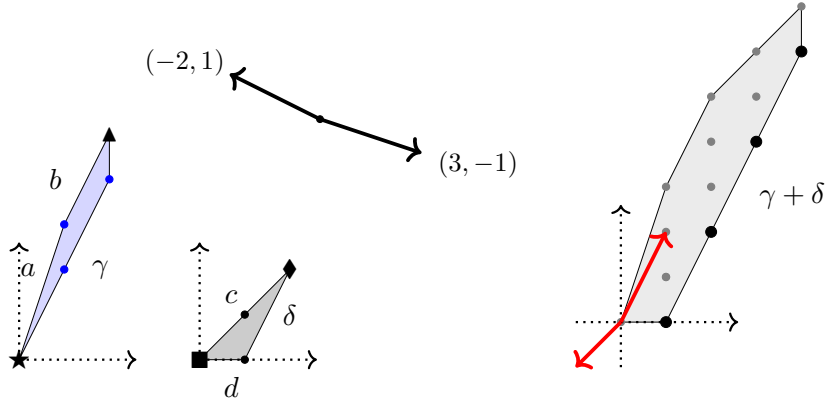

\subsubsection{Mixed volume of finite sets}\label{subs:mixed} Given a convex set $\Delta\subset\R^2$, let $\Vol (\Delta)$ denote its fixed, translation-invariant Lebesgue measure endowed on $\R^2$. Recall that Minkowski's \emph{mixed volume} is the unique real-valued multi-linear (with respect to the Minkowski sum) function of two convex sets $\Delta_1,\Delta_2\subset\R^2$, whose value at two copies of $\Delta$ equals $2\Vol(\Delta)$. It is known that $V(\Delta_1,\Delta_2)$ can be expressed as
\[
\Vol(\Delta_1 +\Delta_2) - \Vol(\Delta_1) - \Vol(\Delta_2).
\] 

One can check that if $\Delta_1 + \Delta_2$ is a line segment, or one of $\Delta_i$ is a point, then $V(\Delta_1,\Delta_2)=0$. The other direction holds true as well, and is a particular case of Minkowski's Theorem for the higher-dimensional mixed volume. We say that $(\Delta_1,\Delta_2)$ is \emph{independent} if $V(\Delta_1, \Delta_2)\neq 0$, or, equivalently, if $\dim (\sum_{i\in I} \Delta_i)\geq |I|$ for all $I\subset {1,2}$. A \emph{dependent} couple above is one that is not independent.

A couple $A$ of finite sets in $\Z^2$ is said to be \emph{independent} if the convex hulls of its members form an independent pair. Similarly, the notation $V(A)$ will refer to the mixed volume of the pair consisting of the convex hulls corresponding to $A$'s members.

\subsection{The number of roots to a system of equations}\label{subs:supp-roots} A bivariate complex polynomial $P\in\C[x_1,x_2]$ is written as a finite linear combination \[\sum_{w\in\N^2} c_wx^w,\] of monomials $x^w:=x_1^{w_1}x_2^{w_2}$, where $w:=(w_1,w_2)$ and $c_w\in\C$. The \emph{support} $\supp P$ of $P$ is the set $\{w\in\Z^2~|~c_w\neq 0\}$.

A pair $\varphi:=(\varphi_1,\varphi_2)$ of polynomials in $\C[x_1,x_2]$ identify a polynomial map $\varphi:\C^2\to\C^2$. The \emph{support} of $\varphi$ is the pair $(\supp\varphi_1,\supp\varphi_2).$ The polynomial system
\begin{equation}\label{eq:sys:phi}
  \begin{array}{@{}ccccc@{}}
 \varphi_1 (x_1,x_2)& = & 0,\\
  \varphi_2 (x_1,x_2)& = & 0,
  \end{array}
\end{equation} is denoted by $\varphi = 0$ and its set of solutions $(x_1,x_2)$ in $\C^2$ is denoted by $\mathbb{V}(\varphi)$. We use $\mathbb{V}^\circ(\varphi)$ to denote the subset $\mathbb{V}(\varphi)\cap\TT$.
We will abuse notations by writing $\#\mathbb{V}(\varphi)$ and $\#\mathbb{V}^\circ(\varphi)$ in reference to the number of isolated solutions to~\eqref{eq:sys:phi} counted with multiplicities in $\C^2$ and $\TT$ respectively.

A consequence of Bernstein's Theorem A is 
\[\#\mathbb{V}^\circ(\varphi)\leq V(\supp \varphi).\] In particular, the set of isolated solutions is empty if the support is dependent.

Another result of Bernstein will be used throughout this paper goes as follows. 

Given a face $\Gamma\prec\supp \varphi$, the restriction of the polynomial $\varphi_i$ ($i=1,2$) to monomial terms $c_wx^w$ satisfying $w\in\Gamma_i$, is denoted by $\varphi_{i,\Gamma}$. We write $\varphi_\Gamma$ to denote the pair $(\varphi_{1,\Gamma},\varphi_{2,\Gamma})$.


\begin{theorem}[Theorem B of~\cite{Ber75} for $n=2$]\label{th:BerB}
Let $\varphi$ be a pair of bivariate polynomials. Then, all solutions to $\varphi = 0$ are isolated and 
\[\#\mathbb{V}^\circ(\varphi) = V(\supp\varphi)\]
if and only if for any non-trivial face $\Gamma\prec\supp\varphi$, we have $\#\mathbb{V}^\circ(\varphi_\Gamma)=\emptyset$.
\end{theorem}

\subsection{Toric change of variables}\label{subs:toric} Following~\cite{Ber75}, a \emph{unimodular toric} change of variables on $x=(x_1,x_2)$ is a map $\TT\to\TT$, $x\mapsto z =(z_1,z_2)$, such that \[x_1=z_1^{u_{11}}z_2^{u_{21}},\quad x_2=z_1^{u_{12}}z_2^{u_{2}},\] where $ U = \big(\begin{smallmatrix} u_{11} & u_{1,2}\\ u_{2,1} & u_{2,2} \end{smallmatrix}\big) \in \SL(2,\Z)$. This transformation induces an isomorphism between polynomials, which by abuse of notation, we also denote by $U$; it is 
\[U:~\C[x_1,x_2]\to\C[z_1^{\pm 1},z_2^{\pm 1}],\] taking the monomial $x^a$ to $x^{Ua}$, where $a:=(a_1,a_2)$. 

Hence, for any polynomial $P\in\C[x_1,x_2]$, we have \[\supp(U P)=U(\supp P),\] and the zero locus $\mathbb{V}^\circ(P)$ is isomorphic to $\mathbb{V}^\circ(U P)$. 
\begin{remark}\label{rem:toric-preserv}
For $U\varphi:=(U\varphi_1,U\varphi_2)$, we have \[\#\mathbb{V}^\circ(\varphi)=\#\mathbb{V}^\circ(U\varphi).\] 
\end{remark}

\begin{example}\label{ex:transf1}
Let $\varphi$ be the pair of polynomials 
\[
  \begin{array}{@{}ccccc@{}}
 \big(1+2uv^2 + 3u^2v^4 + 4uv^3 +5u^2v^5,~-1-2u - 3u^2v^2\big)
  \end{array}
\] supported on the pair of sets in Figure~\ref{fig:tuples+basis} on the left. If $ U = \big(\begin{smallmatrix} 3 & -1\\ -2 & 1 \end{smallmatrix}\big)$, then $U\varphi$ is written as
\[
  \begin{array}{@{}ccccc@{}}
\big( 1+2s + 3s^2 + 4s^2t +5s^3t,~-1-2s^{-1}t^{-2} - 3t^{-1}\big)
  \end{array}
\] 
\end{example}
A transformed pair $U\varphi$ may have monomials with negative exponents. We can transform it to a pair of polynomials by multiplying by a suitable monomial. We denote by $\overline{U}\varphi$ the pair $(\overline{U}x^{r_1}\varphi_1,\overline{U}x^{r_2}\varphi_2)$, where the coordinates of $r_1,r_2\in\N^2$ are the minimal ones that allow to clear denominators. In Example~\ref{ex:transf1}, we have $r_1=(0,0)$ and $r_2 = (1,2)$.
\subsubsection{Base-change using faces}\label{subsub:base-change-faces} We will introduce an above toric change of coordinates $U\in \SL(2,\Z)$ that will help us later deduce a useful description of points in the preimage, under a map, that escape to infinity. It turns out that the best choice for $U$ is one whose entries depend on some faces of supports. 

Let $A=(A_1,A_2)$ be a pair of finite sets in $\N^2$ and let $\Gamma=(\Gamma_1,\Gamma_2)$ be a face of $A$ (see Section~\ref{subsubs:finite-sets}) such that $\dim \Gamma =1$. Our goal is to construct a pair of vectors $\tilde{e}:=(\tilde{e}_1,\tilde{e}_2)$ as follows.

Let $a\in\N^2$ be the closest point in $\Gamma_1+\Gamma_2$ to $(0,0)$. Set $\tilde{e}_1$ to be the primitive integer vector spanning $\Gamma_1+\Gamma_2$ away from $a$ (i.e. the first coordinate of $\tilde{e}_1$ is positive). In the Example of Figure~\ref{fig:tuples+basis}, we have $\gamma +\delta$ is the set of black dots to the right, $a=(1,0)$ and $\tilde{e}_1$ is the red vector $ (1,2)$.

We choose $\tilde{e}_2$ such that

\begin{enumerate}[(i)]

\item\label{it:i} $\tilde{e}$ is the basis of the lattice $\Z^2$.

\item\label{it:ii} The Minkowski sum $A_1 + A_2 +\{-a\}$ is contained in the cone $\{b_1\tilde{e}_1+b_2\tilde{e}_2~|~b_1,b_2\in\R_{\geq 0}\}$.  Or in other words, the basis $\tilde{e}$ spans positively the lattice of points $A_1+A_2-a$.
\end{enumerate} 
This represents a linear transformation $U:\Z^2\to\Z^2$ taking $\tilde{e}$ to the canonical basis of $\Z^2$. In the Example of Figure~\ref{fig:tuples+basis}, $\tilde{e}_2$ is the red vector $(-1,-1)$.

We consider the matrix $T$, where
\[
  E =
  \big(\begin{smallmatrix} \tilde{e}_{11} & \tilde{e}_{12}\\ \tilde{e}_{21} & \tilde{e}_{22} \end{smallmatrix}\big)
  =
  \big(\begin{smallmatrix} v_1 & v_2\\ w_1 & w_2 \end{smallmatrix}\big)
  \in \SL(2,\Z) .
\]
Then, $T$ corresponds to the following change of variables $z = x^{E \binom{1}{1}}$.
Thus the transformation $U$ that we are looking for, is the inverse of this map,
that is
\[
  x_1 \mapsto z_{1}^{w_{2}/D}z_{2}^{-v_{2}/ D},
  x_2 \mapsto z_{1}^{-w_{1}/D}z_{2}^{v_{1}/D},
\]
where $D = \det(E) =  \pm 1$ is the determinant of $T$.

From the properties~\ref{it:i} and~\ref{it:ii} of $T \in \SL(2, \Z)$, we deduce that $E$,
and thus also $U$, depends on
$\Gamma$ and $A$. We denote the subset of $\SL(2,\Z)$ of all such $U$ by
$\Trg$.


 We also have the following immediate consequence. 

 \begin{lemma}\label{lem:vector-change}
   Let $\alpha\in\Z^2$ be the primitive integer vector supporting the  $\Gamma \prec A$.
   Then, for any $U\in\Trg$,  $U\Gamma$ is a face of $UA$.
   Moreover, the vector $U\cdot \alpha^{\Tr}  = (0,1)$ and it supports $U\Gamma$.
 \end{lemma}



 

 \begin{remark}
   \label{rem:one-var}
   For any  $\varphi\in \C[x_1,x_2]^2$, where  $A = \supp(\varphi)$,
   the following hold:
   \begin{itemize}
     \item For any $U\in\Trg$, the system $\ous f_\Gamma =0$ is univariate and $\ous f=0$ is bivariate.
     \item $\ous f_\Gamma =0$ has a solution $\rho \in\C^*$ iff
     $\ous f =0$ has a solution $(\rho,0)\in \C^*\times\{0\}$ iff
     $f_\Gamma = 0$ has a solution in $\TT$.
   \end{itemize}
\end{remark}

 \begin{proposition}\label{prop:main}
 Let $\varphi$ be a pair of bivariate polynomials. Let $\mathcal{S}$ denote the set of non-trivial faces $\Gamma\prec\supp \varphi$ for which there exists a matrix $U\in\Trg$, such that $\ous\varphi = \underline{0}$ has $m_\Gamma>0$ solutions in $\C^*\times\{0\}$, counted with multiplicities. Then, we have 
\begin{equation}\label{eq:sol-numb}
\#\mathbb{V}^\circ(\varphi) = V(\supp\varphi ) -\sum_{\Gamma\in\mathcal{S}} m_\Gamma.
\end{equation} 
\end{proposition}

\begin{proof}
Let $A$ denote $\supp\varphi$. We proceed using similar arguments as in the proof of~\cite[Theorem B]{Ber75}. Consider the parametrized polynomial system\begin{equation}\label{eq:parametr}
\varphi + t\psi = 0,
\end{equation}  where $\psi := (\psi_1,\psi_2)\in\C[A]$ and $t\in ]0,1[$. Using~\cite[Theorem A]{Ber75}, one can choose the pair $\psi$ so that~\eqref{eq:parametr} has $V(A)$ parametrized isolated solutions $x(t)\in\TT$. The result follows by proving the claim: \emph{There exists a bijection $x(t)\mapsto\rho$ between the set of solutions to~\eqref{eq:parametr} in $\TT$, escaping $\TT$ as $t\rightarrow 0$, and the multiset defined as the union of all solutions in $\C^*\times {0}$, counted with multiplicities, to systems $\ous\varphi = 0$ above.}

A solution $x(t)$ to~\eqref{eq:parametr} can be presented as a function in $t$, where its $j$-th coordinate ($j=1,2$) is written as the Puiseux series 
\[
a_jt^{\alpha_j}+~\text{higher order terms in $t$},
\] where $a_j\in\C^*$ and $\alpha_j\in\Q$. 

If we plug $x(t)$ into $\varphi_i + t\psi_i$ and set the coefficient of the smallest power $p$ of $t$ equal to zero, we obtain $\varphi_{i,\Gamma}(a)=0$ for some $\Gamma\prec A$. Indeed, the value $p$ is the minimum $\min(\langle\alpha,q\rangle~|~q\in A_i)$, which is reached only for points $q$ in the face of $A_i$ supported by the vector $\alpha:=(\alpha_1,\alpha_2)$. Therefore, $a:=(a_1,a_2)\in\TT$ is a solution to 
\[\varphi_\Gamma=0.\] 
Assuming in what follows that $\#\mathbb{V}^{\circ}<V(A)$ (the result follows automatically from Theorem~\ref{th:BerB} if we do not make this assumption). 

Let us choose one $x(t)$ above that escapes $\TT$ for small enough $t$. 

On the one hand, the vector $\beta:=U\cdot\alpha^{\tr}$ supporting the face $U\Gamma$ of $UA$ is equal to $(0,1)$ (see Lemma~\ref{lem:vector-change}). On the other hand, the point $z(t)$, defined as 
\begin{equation}\label{eq:zt}
z^U(t)=x(t)
\end{equation}is a solution to \begin{equation}\label{eq:param-trsf}
\ous(\varphi + t\psi) = 0.
\end{equation} Then, for $j=1,2$, we have 
\[
z_j(t)=b_jt^{\beta_j}+~\text{higher order terms in $t$},
\] for some $b:=(b_1,b_2)\in\TT$. Therefore, the limit $z(0)$ belongs to $\C^*\times\{0\}$, and is a solution to $\ous\varphi = 0$.

This description implies the following: Each $x(t)$ escaping $\TT$ determines a unique proper face $\Gamma\prec A$, together with a matrix $U\in\Trg$ such that $\ous\varphi=0$ has a solution $\rho\in\C^*\times\{0\}$ satisfying $\rho\in\C^*\times\{0\}$, and $\rho =\lim_{t\rightarrow 0} z(t)$, where $z(t),x(t)$ satisfy~\eqref{eq:zt}. 

The multiplicity $m_\rho$ of $\rho$ is no less than the number $N_\rho$ of distinct solutions $z(t)$ to~\eqref{eq:param-trsf}, converging to $\rho$. 

Finally, our choice of $\psi$ in the beginning implies that for any $t$, the system
\[
\ous(\varphi +t\psi)_\Gamma=0
\] has no solutions in $\TT$ for any proper $\Gamma$. Then,~\eqref{eq:param-trsf} has no solutions in $\C^*\times\{0\}$ for any matrix $U\in\Trg$. Therefore, we have $m_\rho = N_\rho$.
\end{proof}

\section{Relevant faces and proof of Theorem~\ref{th:family}}\label{sec:relevant} One does not require the data of all monomial terms appearing in a polynomial map in order to compute its isolated missing points. In this section, we point out
those faces that are pertinent to such computation. We will utilize all notations and Proposition~\ref{prop:main} in the previous section to prove some important technical results. These will be used to prove Theorem~\ref{th:family} and the results in Sections~\ref{sec:proof-prop1} --~\ref{sec:proof-prop3}.

Consider a dominant polynomial map $f:=(f_1,f_2):\C^2\to\C^2$ and let $A=(A_1,A_2)$ denote the pair $\supp f$. For any $y=(y_1,y_2)\in\C^2$, we denote by $f-y$ the pair $(f_1-y_1,f_2-y_2)$, and by $f-y = 0$ the system
\[
  \begin{array}{@{}ccccc@{}}
f_1(x_1,x_2) - y_1& = & 0,\\
f_2(x_1,x_2) - y_2& = & 0.
  \end{array}
\] 

We are interested in describing the set $\C^2\setminus f(\C^2)$. This problem is invariant under translation $f+c$ for any $c\in\C^2$. Therefore, we will assume in what follows that $f(0,0)\in\TT$. This implies that each member of $A$ contains $(0,0)$ (see e.g. Example~\ref{ex:transf1} and Figure~\ref{fig:tuples+basis}).

\begin{remark}\label{rem:dom-indep}
Note that $f$ maps $\C^2$ onto a line if $A_i=\{(0,0)\}$ for some $i\in\{1,2\}$, and onto a curve if $\dim A_1 + A_2 = 1$. Therefore, it is necessary for $A$ to be independent (see Section~\ref{subs:mixed}) in order for $f$ to be dominant.
\end{remark}

\subsection{Face-classification and generic non-properness}\label{subs:gen-non} We will distinguish several types of faces of $A$ (see the diagram in Figure~\ref{fig:diagram}).

\begin{center}
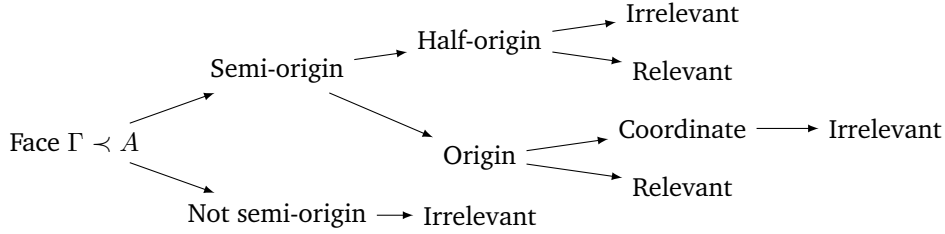
\begin{figure}[h]
\begin{tikzpicture}
  [
    grow                    = right,
    sibling distance        = 6em,
    level distance          = 5em,
    edge from parent/.style = {draw, -latex},
    every node/.style       = {font=\footnotesize},
    sloped
  ]
  \node [root]{Face $\Gamma\prec A$}
       child[sibling distance  = 5em,level distance  = 7em]{ node [env] {Not semi-origin}
		child[sibling distance  = 5em,level distance  = 7em]{ node [env] {Irrelevant}}       
       }
     child[sibling distance  = 5em,level distance  = 7em]{ node [env] {Semi-origin} 
        	child[sibling distance  = 6em,level distance  = 7em]{ node [root] {Origin} 
        child[sibling distance  = 2em,level distance  = 7em]{ node [env] {Relevant}}    
	    child[sibling distance  = 2em,level distance  = 7em]{ node [env] {Coordinate}
	    		    child[sibling distance  = 2em,level distance  = 7em]{ node [env] {Irrelevant}}}       
        }
        	child[sibling distance  = 2em,level distance  = 7em]{ node [root] {Half-origin} 
        child[sibling distance  = 2em,level distance  = 7em]{ node [env] {Relevant}}    
	    child[sibling distance  = 2em,level distance  = 7em]{ node [env] {Irrelevant}}       
        }};
\end{tikzpicture}
\caption{Diagram classifying types of faces}\label{fig:diagram}
\end{figure}
\end{center}

\begin{definition}\label{def:face-non-prop}
Let $A$ be an independent pair of finite sets in $\N^2$. A face $\Gamma = (\Gamma_1,\Gamma_2)\prec A$ is 

\begin{itemize}

	\item \emph{semi-origin} if at least one of its members contains the origin, that is $(0,0)\in\Gamma_1$, or $(0,0)\in\Gamma_2$.
	
		\begin{itemize}

			\item[-] $\Gamma$ is \emph{origin} if both its members contain $(0,0)$.		
			\item[-] $\Gamma$ is \emph{half-origin} otherwise.		
		
		\end{itemize}
		
	\item \emph{coordinate} if all of its supporting vectors $\alpha = (\alpha_1,\alpha_2)$ (see Section~\ref{subs:faces}) have one coordinate that is positive and another that is zero.
	
	\item \emph{relevant} if it is semi-origin, not coordinate, and no member of $\Gamma$ is a point different from $(0,0)$.
	\item \emph{irrelevant} if it is not relevant.
\end{itemize}
\end{definition}

\begin{example}\label{ex:non-prop}
The pair of sets on the left of Figure~\ref{fig:tuples+basis} give the following classification for its faces: 
\begin{itemize}

	\item Origin faces: $(\bigstar,\blacksquare)$, $(a,\blacksquare)$, $(\bigstar,d)$ 

	\item Half-origin faces: $(b,\blacksquare)$, $(\blacktriangle,\blacksquare)$, $(\blacktriangle,c)$, $(\gamma,\delta)$, $(\bigstar,\circ)$	

	\item Coordinate faces: $(\bigstar,d)$	
	
	\item Relevant faces: $(a,\blacksquare)$, $(b,\blacksquare)$, $(\bigstar,\blacksquare)$, $(\gamma,\delta)$
	
	\item The face $(\blacktriangle,c)$ is semi-origin, not coordinate, but not relevant. This is since $\dim\blacktriangle = 0 $, but $\blacktriangle\neq \{(0,0)\} $. The same goes for $(\bigstar,\circ)$\end{itemize}
\end{example} Recall the notation in Section~\ref{sec:prelim-results}.

\begin{lemma}\label{lem:not-coordinate}
If $\#\mathbb{V}(f)<V(A)$, then there exists a face $\Gamma\prec A$ that is not coordinate, such that the system $
f_\Gamma = 0$ has a solution in $\TT$.
\end{lemma} 

\begin{proof}
Split the isolated points in $\#\mathbb{V}(f)$ into two subsets: Those in $\TT$ and those outside it. 

The former has size equal to
\begin{equation}\label{eq:size1}
V(A) -\sum m_\Gamma,
\end{equation} for all $\Gamma$ as in Proposition~\ref{prop:main} and the latter has size equal to 
\begin{equation}\label{eq:size2}
\sum_{\substack{\Gamma \prec A\\ \Gamma~\text{coordinate}}} m_\Gamma.
\end{equation} 
We sum up~\eqref{eq:size1} and~\eqref{eq:size2} to deduce that $m_\Gamma >0$ for some $\Gamma\prec A$ that is not coordinate. This yields the proof.
\end{proof}

Recall Definition~\ref{def:gen-prop} for generically non-proper maps.

\begin{lemma}\label{lem:gen-non-degree}
Let $f$ be generically non-proper. Then, we have 
\[\mu_f = V(A).\]
\end{lemma}

\begin{proof} 
We have $\# f^{-1}(0) = V(A)$ from Definition~\ref{def:gen-prop}, $\mu_f\leq V(A)$ from Theorem~\ref{th:BerB} and $\# f^{-1}(0) \leq \mu_f$ from $f$ being dominant. 
\end{proof}

\begin{lemma}\label{lem:solutions-infty1} 
Let $f$ be generically non-proper. If for some face $\Gamma\prec A$, there exists $y\in \C^2$ such that the system 
$(f-y)_\Gamma = 0$ has a solution in $\TT$, then $\Gamma$ is a semi-origin face. 
\end{lemma} 

\begin{proof}
Let $y$ be any point in $\C^2$, and let $\Gamma$ be a face of $A$ that is not semi-origin. Then, we have $(f-y)_\Gamma = f_\Gamma$. Therefore, if $(f-y)_\Gamma = 0$ has a solution in $\TT$, then so does $f_\Gamma = 0$. 

We deduce from Theorem~\ref{th:BerB} that $\#\mathbb{V}^\circ(f)<V(A)$. This contradicts Lemma~\ref{lem:gen-non-degree}. 
\end{proof}

Recall Definition~\ref{def:Jel} for the Jelonek set $\mathcal{J}_f$ of $f$. The following result will be used in Section~\ref{sec:proof-prop2}.

\begin{lemma}\label{lem:solutions-infty2} 
Let $f$ be generically non-proper. Then, we have $y\in\mathcal{J}_f$ if and only if for some relevant $\Gamma\prec A$, the system $ (f-y)_\Gamma = 0 $ has a solution in $\TT$. 
\end{lemma} 

\begin{proof}
For the first direction, assume that $y\in\mathcal{J}_f$. Then, $\#\mathbb{V}(f-y)<\mu_f$. Lemma~\ref{lem:not-coordinate} implies that there exists a non-coordinate face $\Gamma\prec A$ such that $ (f-y)_\Gamma = 0 $ has a solution in $\TT$. We can deduce that $\Gamma$ is relevant from Lemma~\ref{lem:solutions-infty1}.

We prove the second direction. Let $y$ be a point outside $\mathcal{J}_f$. Then, Lemma~\ref{lem:gen-non-degree} shows that 
\begin{equation}\label{eq:volume:roots}
V(A) = \#\big(\mathbb{V}(f-y)\setminus\mathbb{V}^\circ(f-y) \big) + \#\mathbb{V}^\circ(f-y).
\end{equation}

Assume that $y\neq f(0,0)$. Then, using the notations of Proposition~\ref{prop:main}, the first summand in~\eqref{eq:volume:roots} is accounted for by $\sum m_\Gamma$, where $\Gamma$ runs over all (at most two) coordinate faces of $A$. Equation~\eqref{eq:sol-numb} shows that for any other faces $\Gamma$, we have $m_\Gamma =0$. Therefore, the system $ (f-y)_\Gamma = 0 $ has no solutions if $\Gamma$ is relevant.

Finally, one can check that $y= f(0,0)\wedge y\notin\mathcal{J}_f\Rightarrow A$ has no relevant faces.
\end{proof}

Recall the set $\mathcal{K}_f$ defined in Section~\ref{subsec:th11}. The following result will be used in Sections~\ref{sec:proof-prop1} --~\ref{sec:proof-prop3}.

\begin{lemma}\label{lem:Jel-discr}
Let $f$ be generically non-proper. Then, we have $y\in \mathcal{K}_f$ if and only if there exists a relevant face $\Gamma\prec A$ and a matrix $U\in\Trg$ such that 
$\ous(f-y) = 0$
has a solution in $\C^*\times\{0\}$ of multiplicity $\geq 2$.
\end{lemma}

\begin{proof}
Let $y\in\mathcal{K}_f$. Then, there exists a sequence $\{x_k\}_{k\geq 0}\subset C(f)\subset\C^2$ that converges to infinity, and $f(x_k)$ converges to $y$~\cite{Jel93}. In particular, this is a sequence of double roots to $
f - f(x_k) = 0$, converging to infinity.

Note that any toric transformation $U$ preserves the number of solutions in $\TT$ counted with multiplicities (Remark~\ref{rem:toric-preserv}). Therefore, using the curve selection Lemma, we follow closely the proof of Proposition~\ref{prop:main} to deduce the following claim:

There exists a sequence $\{x_k\}_{k\geq 0}\subset\C^2$ of double roots to 
\[
\ous(f - f(x_k)) = 0,
\] converging to a double root $\rho\in\C^*\times\{0\}$ to $\ous(f-y) = 0$. This proves the first direction.

We omit the proof of the other direction since it is similar to the first one. 
\end{proof}

\subsection{Proof of Theorem~\ref{th:family}}\label{subsec:Th3-proof} Recall that each member of the pair $A$ contains $(0,0)$. For $i=1,2$, let $\C[A_i]\cong\C^{|A_i|}$ denote the space of all polynomials $P:\C^2\rightarrow\C$ such that $\supp P \subset A_i$ (a polynomial becomes identified with its coefficients). The space $\C[A_1]\oplus \C[A_2]$ is denoted by $\C[A]$.

\begin{example}\label{ex:map-space}
Let $A$ be as in Figure~\ref{fig:ex1} defined by the pair $
  A_1 =
 \left\lbrace \big(\begin{smallmatrix} 0 \\ 0 \end{smallmatrix}\big),\big(\begin{smallmatrix} 1 \\ 1 \end{smallmatrix}\big), \big(\begin{smallmatrix} 1 \\ 2 \end{smallmatrix}\big)\right\rbrace$,\\
  $A_2=\left\lbrace \big(\begin{smallmatrix} 0 \\ 0 \end{smallmatrix}\big),\big(\begin{smallmatrix} 1 \\ 0 \end{smallmatrix}\big), \big(\begin{smallmatrix} 1 \\ 1 \end{smallmatrix}\big)\right\rbrace$. Then, any complex map 
\begin{equation}\label{eq:map:ex1}
  \begin{array}{@{}ccccl@{}}
   (u,v) & \mapsto & (a_0 +a_1uv + a_2uv^2,~b_0 +b_1u+b_2uv),
  \end{array}
\end{equation} is identified with $(a_0,a_1,a_2,b_0,b_1,b_2)\in\C[A]\cong\C^6$.
\end{example}

\begin{figure}[h]
\centering
\begin{tikzpicture}
    \tikzstyle{bluefill1} = [fill=blue!20,fill opacity=0.8]          
    \tikzstyle{blackfill1} = [fill=gray,fill opacity=0.4]          
    \tikzstyle{greyfill1} = [fill=gray!20,fill opacity=0.8]          
    \tikzstyle{conefill} = [pattern = north east lines, pattern color=gray]          
    \tikzstyle{ann} = [fill=white,font=\footnotesize,inner sep=1pt] 
    \tikzstyle{ghostfill} = [fill=white]	
    \tikzstyle{ghostdraw} = [draw=black!50]					
	\tikzstyle{ann1} = [font=\footnotesize,inner sep=1pt] 
	\tikzstyle{ann2} = [font=\normalsize,inner sep=1pt] 


\begin{scope}[xshift = -5cm, yshift = 0cm, xscale=1.2, yscale=1.2]


	\draw[arrows=->,line width=0.3 mm, dotted] (0,0)-- (0,1.3); 
	
	\draw[arrows=->,line width=0.3 mm, dotted] (0,0)-- (1.3,0); 


	\filldraw[bluefill1,line width=0.0 mm ] (0,0)--(0.5,0.5)--(0.5,1)-- (0,0);	

		\draw [fill, color = blue] (0,0) circle [radius=0.04];
		\draw [fill, color = blue] (0.5,0.5) circle [radius=0.04];
		\draw [fill, color = blue] (0.5,1) circle [radius=0.04];

 \end{scope}	
	

\begin{scope}[xshift = 0cm, yshift = 0cm, xscale=1.2, yscale=1.2]

	\draw[arrows=->,line width=0.3 mm, dotted] (0,0)-- (0,1.3); 
	
	\draw[arrows=->,line width=0.3 mm, dotted] (0,0)-- (1.3,0); 


	\filldraw[blackfill1,line width=0.0 mm ] (0,0)--(0.5,0)--(0.5,0.5)-- (0,0);	

		\draw [fill, color = black] (0,0) circle [radius=0.04];
		\draw [fill, color = black] (0.5,0) circle [radius=0.04];
		\draw [fill, color = black] (0.5,0.5) circle [radius=0.04];
\end{scope}	
\end{tikzpicture}
\caption{Support of the map in Example~\ref{ex:map-space}}\label{fig:ex1}
\end{figure}
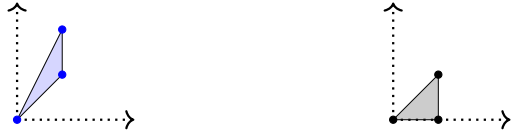

\begin{proposition}\label{prop:open-dense}
Let $A$ be an independent pair of finite subsets in $\N^2$. If $A$ has no relevant faces, then proper maps form a dense subset in $\C[A]$, and no map in $\C[A]$ is generically non-proper. Otherwise, all maps in $\C[A]$ are non-proper, and generically non-proper maps form a dense subset of $\C[A]$.
\end{proposition}

\begin{proof}
For any face $\Gamma\prec A$, let $\Res_\Gamma(A)$ denote the \emph{multivariate resultant of $\Gamma$} in $\C[A]$. This is, the subset of all maps $f\in \C[A]$ such that $f_\Gamma = 0$ has a solution in $\TT$. Multivariate resultants were introduced in~\cite{GKZ94}, and it was shown in~\cite[Theorem 1.1]{St94} that $\Res_\Gamma(A)$ is a Zariski closed subset of $\C[A]$.

For any $f\in \C[A]$ and any $y\in\mathcal{J}_f$, we have $\#\mathbb{V}(f-y)< V(A)$. Then, Lemma~\ref{lem:not-coordinate} shows that there exists a non-coordinate face $\Gamma\prec A$ such that $(f-y)_\Gamma = 0$ has a solution in $\TT$.

If $A$ has no relevant faces, then either $\mathcal{J}_f = \emptyset$, or $(f-y)_\Gamma = f_\Gamma = 0$. In the firt case, $f$ is proper (in particular, not generically non-proper). In the second case, we have $f\in\Res_\Gamma(A)$ and Lemma~\ref{lem:solutions-infty1} shows that $f$ is not generically non-proper.

Assume that $A$ has a relevant face $\Gamma$. Then, for any $f\in\C[A]$, the set $f^{-1}(f(0,0))$ contains one of the two coordinate axes of $\C^2$. We deduce that any $f$ is non-proper. 

To show that generically non-proper maps form a dense subset of $\C[A]$, we proceed as follows.

Let $|\Jac f|$ denote the determinant of the Jacobian matrix of $f$ and let $\Sigma$ denote its support. Clearly, we have $B:=(A_1,\Sigma)$ is independent of $f$ as long as $f$ is generically chosen. 

The coefficients of $|\Jac f|$ are polynomials in the coefficients of $f$. Moreover, recall that for any face $\Lambda\prec B$, the set $\Res_\Lambda(B)$ is a Zariski closed subset of $\C[B]$. Therefore, the set 
\[
\mathcal{S}_\Lambda:=\left\lbrace (f_1,f_2)\in\C[A]~|~(f_1,|\Jac f|)\in \Res_\Lambda(B)\right\rbrace
\] is a Zariski closed subset of $\C[A]$. Note that the same holds true if we replace $B$ by $(A_2,\Sigma)$.

Finally, we conclude from the definition of generic non-properness (Definition~\ref{def:gen-prop}) and Proposition~\ref{prop:main} that it enough for $f\in\C[A]$ to be outside all subsets $\Res_\Gamma(A)$ and $\mathcal{S}_\Lambda$ for it to be generically non-proper.
\end{proof}

Consider the pair $A$ in Figure~\ref{fig:ex1} corresponding to the map $f$ in Example~\ref{ex:map-space}. For any choice of coefficients, $f$ is non-proper since $\{u=0\}\subset f^{-1}(a_0,b_0)$. Indeed, $A$ has a relevant face; it is supported by $(2,-1)$. 

If $(\gamma,\delta)\prec A$ is supported by the vector $(-1,0)$, then, $f_{(\gamma,\delta)}(u,v) = 0$ has solutions in $\TT$ if and only if $a_1b_2 - b_1a_2 = 0$. Therefore, the map is generically non-proper inside the dense subset $\{a_1b_2 - b_1a_2 \neq 0\}$ of $\C^6$.
\section{Long faces and proof of Proposition~\ref{prop:S0f}}\label{sec:proof-prop1} In this section, we prove  Proposition~\ref{prop:S0f}. To do this, we need a few extra technical results. These are gathered in Section~\ref{subs:techn}.

\subsection{Lengths of faces}\label{subs:techn} In this part, we will introduce two invariants for faces of finite pairs in $\Z^2$: \emph{integer length} and \emph{dimensional length}.

\subsubsection{Integer length}\label{subsub:integ-leng} For any set $\sigma\subset \R^2$, such that $\dim \conv(\sigma) = 1$, we use $\ell(\sigma)$ to denote the number $|\conv(\sigma)\cap \Z^2| - 1$, where $\conv(\cdot)$ denotes the convex hull of any set in $\R^n$. 
\begin{lemma}\label{lem:mixed-segments}
Let $S_1,S_2\subset\R^2$ be two bounded segments having rational slopes. Then, we have 
\[
\ell(S_1)\cdot\ell(S_2)\leq V(S_1,S_2).
\]
\end{lemma}

\begin{proof}
From $\dim S_1 = \dim S_2=1$, we have $\Vol(S_1)=\Vol(S_2)=0$, and thus $V(S_1,S_2) = \Vol(S_1+S_2)$. Let $L_\sigma(S_2)$ be the union of lines in $\R^2$, parallel to $S_1$, and intersecting $S_2$ at $S_2\cap\N^2$ and $L_\theta(S_1)$ be the analogous union of lines, but with $S_1$ and $S_2$ switched. Then, the set $L_\theta(S_1)\cup L_\sigma(S_2)$ subdivides $S_1+S_2$ into $\ell(S_1)\cdot\ell (S_2)$ parallelograms $P_{ij}$, $i=1,\ldots,\ell(S_1)$, $j=1,\ldots,\ell(S_2)$, from which we obtain \[\Vol(S_1+S_2) = \sum_{i,j}\Vol(P_{ij}).\] From the construction of $P_{ij}$, for any $i,j$ we get $\Vol P_{ij} =|\det (\alpha,\beta)|$, where $\alpha$ and $\beta$ are the directional vectors for $S_1$ and $S_2$ respectively. The result follows from $|\det (\alpha,\beta)|\in \N^*$.
\end{proof} Let $f:\C^2\to\C^2$ be a dominant polynomial map, and let $A$ denote the support of $f$ (See Section~\ref{subs:supp-roots}). Recall that, for each face $\Gamma\prec A$, $\dim\Gamma = 1$, we have $\Trg$ refers to the space of transformations in $\SL(2,\Z)$, constructed using $\Gamma$ as in Section~\ref{subsub:base-change-faces}. 
\begin{lemma}\label{lem:degree-length}
Let $\Gamma\prec A$ be a relevant face (see Definition~\ref{def:face-non-prop}), and let $U$ be a matrix in $\Trg$. Assume that for some $i\in\{1,2\}$, there exists a subset $\sigma\subset UA_i$ such that $
U(f_i)_\sigma$ is a polynomial $P$ in $z_1$, up to monomial multiplication, with $P(0)=0$. Then, we have 
\begin{equation}\label{eq:degree-length}
\deg P = \ell(\xi)\leq \deg f_i/2,
\end{equation} where $\xi\subset A_i$ and $\sigma = U\xi$.
\end{lemma}

\begin{proof}
The assumptions of the Lemma imply that both $\conv(\sigma)$ and $\conv(\xi)$ have dimension one, and that $\ell(\sigma) = \ell(\xi)$.

To show that the equality part of~\eqref{eq:degree-length}, it suffices to notice that $\deg P = \ell(\sigma)$.

Finally, $\Gamma$ being relevant implies that $\conv(\xi)$ is not horizontal, nor vertical. One can check (Figure~\ref{fig:lemma:deg-length}) that
\[
|\conv(\xi)\cap A_i|\leq\deg f_i/2.
\] 
This finishes the inequality part of~\eqref{eq:degree-length}. 
\end{proof} 

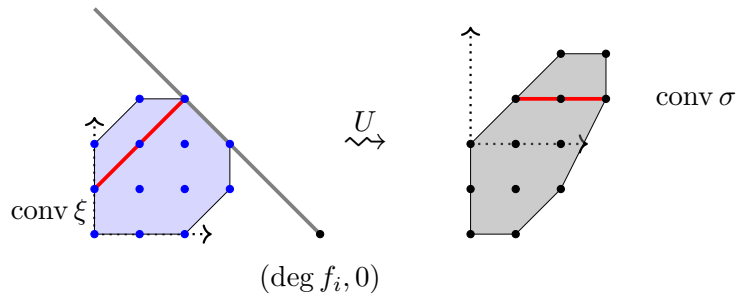
\begin{figure}[h]
\centering
\begin{tikzpicture}
    \tikzstyle{bluefill1} = [fill=blue!20,fill opacity=0.8]          
    \tikzstyle{blackfill1} = [fill=gray,fill opacity=0.4]          
    \tikzstyle{greyfill1} = [fill=gray!20,fill opacity=0.8]          
    \tikzstyle{conefill} = [pattern = north east lines, pattern color=gray]          
    \tikzstyle{ann} = [fill=white,font=\footnotesize,inner sep=1pt] 
    \tikzstyle{ghostfill} = [fill=white]	
    \tikzstyle{ghostdraw} = [draw=black!50]					
	\tikzstyle{ann1} = [font=\footnotesize,inner sep=1pt] 
	\tikzstyle{ann2} = [font=\normalsize,inner sep=1pt] 


\begin{scope}[xshift = -5cm, yshift = 0cm, xscale=1.2, yscale=1.2]


	\draw[arrows=->,line width=0.3 mm, dotted] (0,0)-- (0,1.3); 
	
	\draw[arrows=->,line width=0.3 mm, dotted] (0,0)-- (1.3,0); 


	\filldraw[bluefill1,line width=0.0 mm ](0,0)--(1,0)--(1.5,0.5)
										-- (1.5,1) -- (1,1.5) -- (0.5,1.5)--(0,1)-- (0,0) ;	
										
	\draw[ line width=0.5 mm, color = gray] (2.5,0) -- (0, 2.5); 
	\draw[ line width=0.5 mm, color = red] (0,0.5) -- (1, 1.5); 	
		\draw [fill, color = blue] (0,0) circle [radius=0.04];
		\draw [fill, color = blue] (0,0.5) circle [radius=0.04];
		\draw [fill, color = blue] (0.5,0.5) circle [radius=0.04];
		\draw [fill, color = blue] (0.5,0) circle [radius=0.04];
		\draw [fill, color = blue] (1,0) circle [radius=0.04];
		\draw [fill, color = blue] (0,1) circle [radius=0.04];
		\draw [fill, color = blue] (1,1) circle [radius=0.04];
		\draw [fill, color = blue] (1,0.5) circle [radius=0.04];
		\draw [fill, color = blue] (0.5,1) circle [radius=0.04];
		\draw [fill, color = blue] (0.5,1.5) circle [radius=0.04];
		\draw [fill, color = blue] (1,1.5) circle [radius=0.04];
		\draw [fill, color = blue] (1.5,0.5) circle [radius=0.04];
		\draw [fill, color = blue] (1.5,1) circle [radius=0.04];
		\draw [fill] (2.5,0) circle [radius=0.04];		

	\node[ann2] at (2.5,-0.5)		{$(\deg f_i,0)$};
	\node[ann2] at (-0.5,0.25)		{$\conv \xi$};			
	\node[ann2] at (3,1.25)		{$U$};			
	\node[ann2] at (3,1)		{{\LARGE $\rightsquigarrow$}};

 \end{scope}	
	

\begin{scope}[xshift = 0cm, yshift = 0cm, xscale=1.2, yscale=1.2]


	\draw[arrows=->,line width=0.3 mm, dotted] (0,1)-- (0,2.3); 
	
	\draw[arrows=->,line width=0.3 mm, dotted] (0,1)-- (1.3,1); 
	

	\filldraw[blackfill1,line width=0.0 mm ](0,0)--(0.5,0)--(1,0.5)
										-- (1.5,1.5) -- (1.5,2) -- (1,2)--(0,1)-- (0,0) ;	



	\draw[ line width=0.5 mm, color = red] (0.5,1.5) -- (1.5, 1.5); 	
		\draw [fill, color = black] (0,0) circle [radius=0.04];
		\draw [fill, color = black] (0,0.5) circle [radius=0.04];
		\draw [fill, color = black] (0,1) circle [radius=0.04];
		\draw [fill, color = black] (0.5,0) circle [radius=0.04];
		\draw [fill, color = black] (0.5,0.5) circle [radius=0.04];
		\draw [fill, color = black] (0.5,1) circle [radius=0.04];
		\draw [fill, color = black] (0.5,1.5) circle [radius=0.04];
		\draw [fill, color = black] (1,0.5) circle [radius=0.04];
		\draw [fill, color = black] (1,1) circle [radius=0.04];
		\draw [fill, color = black] (1,1.5) circle [radius=0.04];
		\draw [fill, color = black] (1,2) circle [radius=0.04];
		\draw [fill, color = black] (1.5,1.5) circle [radius=0.04];
		\draw [fill, color = black] (1.5,2) circle [radius=0.04];

		\node[ann2] at (2.5,1.5)		{$\conv \sigma$};

\end{scope}	
\end{tikzpicture}
\caption{An instance of Lemma~\ref{lem:degree-length}.}\label{fig:lemma:deg-length}
\end{figure}

\subsubsection{Dimensional length}\label{subsub:dimen-leng} We also have the following refinement for types of faces of $A$. We say that a face $\Gamma\prec A$ is \emph{long} if $\dim\Gamma_1 =\dim\Gamma_2 = 1$ and \emph{short} otherwise.

Recall that a supporting vector $\alpha=(\alpha_1,\alpha_2)\in\R^2$ of a relevant face of $A$ satisfies $\alpha_1\alpha_2<0$. If $\alpha_1>0$, we say that it is a \emph{left} relevant face and it is a \emph{right} one otherwise.

In Figure~\ref{fig:tuples+basis} to the left, faces $(a,\blacksquare)$ and $(b,\blacksquare)$ of the pair of sets appearing are left and are short. The face $(\gamma,\delta)$, is right and long.

\begin{lemma}\label{lem:le-ri-lo-sh(1)}
Assume that the map $f$ is generically non-proper. Then, its support $A$ has at most one right/left long relevant face.
\end{lemma}

\begin{proof}
Assume without loss of generality that $\Gamma$ is right, and that the first member $\Gamma_1\subset A_1$ of $\Gamma$ contains $(0,0)$ (see Definition~\ref{def:face-non-prop}). If $\alpha=(\alpha_1,\alpha_2)\in\Q^2$ denotes the supporting vector of $\Gamma_1$, then $(\alpha_1,\alpha_2+q)$ supports the vertex $(0,0)\in A_1$ for any $q>0$. This implies that any other right bad face of $A$ has to be a short one. 
\end{proof}

\subsection{Proof of Proposition~\ref{prop:S0f}}\label{subs:proof-S0} Assume in what follows that $f$ is generically non-proper. We will use $\mathcal{I}_\emptyset:= \Em$ to denote the set of isolated points in $\C^2\setminus f(\C^2)$, and recall the sets $\mathcal{C}^\circ:=\mathcal{C}^\circ_f$ and $\mathcal{K}:=\mathcal{K}_f$ defined in Section~\ref{subsec:th11} for dominant maps $f:\C^2\to\C^2$. 

 We start by proving the following Lemma.

\begin{lemma}\label{lem:simple-transf-sol}
For any $q\in\mathcal{I}_\emptyset\cap\mathcal{C}^\circ\setminus\mathcal{K}$, one of the following claims holds true.

\begin{itemize}
	\item There exists a long origin relevant face $\Gamma\prec A$ and a matrix $U\in\Trg$ such that 
	\[
	\ous(f-q)_\Gamma=0
	\] has $\mu_f$ distinct solutions in $\C^*$.
	
	\item There exist two long relevant faces $\Gamma(1),\Gamma(2)\prec A$, together with two matrices $U_1\in\mathsf{TM}_{\Gamma(1)}(\Z^2)$ and $U_2\in \mathsf{TM}_{\Gamma(2)}(\Z^2)$, such that
\[
\overline{U}_i(f - q)_{\Gamma(i)} = 0
\] has $N_i$ solutions $(i=1,2)$ in $\C^*$, and $N_1 +N_2 = \mu_f$.
\end{itemize}

\end{lemma}

\begin{proof} The system $f-q =0$ has no solutions in $\C^2$. Then, Theorem~\ref{th:BerB} and Remark~\ref{rem:one-var} show that $\ous(f-y) = 0$ has a solution in $\C^*\times\{0\}$. 

The proof will follow from Proposition~\ref{prop:main} and Lemma~\ref{lem:le-ri-lo-sh(1)} once we show the following claim:

\textit{All such faces $\Gamma$ can only be long and relevant. Moreover, solutions to $\ous(f-y) = 0$ in $\C^*\times\{0\}$ are simple.}

Since $f-q= 0$ has no solutions in $\C^2$ (in particular, no solutions in $\C^2\setminus\TT$), $\Gamma$ cannot be a coordinate face. Then, Lemma~\ref{lem:solutions-infty1} yields that $\Gamma$ is pertinent.

To deduce that $\Gamma$ is long, notice that $\Gamma$ short $\Rightarrow$ $f_i(0,0) = q_i$ (for some $i\in\{1,2\}$) $\Rightarrow$ $q\notin\mathcal{C}^\circ$.

Finally, we have $q\notin\mathcal{K}\Rightarrow$ all solutions to $\ous(f-y) = 0$ are simple (see Lemma~\ref{lem:Jel-discr}). This proves the two statements of the claim.
\end{proof} Assume that $\mathcal{I}_\emptyset\cap\mathcal{C}^\circ \setminus \mathcal{K}\neq\emptyset$, and consider the set
\[
\{q_1,\ldots,q_N\}:=\mathcal{I}_\emptyset\cap\mathcal{C}^\circ \setminus \mathcal{K}.
\] Lemma~\ref{lem:simple-transf-sol} shows that $A$ has a long relevant face $\Gamma$. We decompose $\{1,\ldots,N\}$ into a disjoint union $I_1\sqcup\cdots\sqcup I_\lambda$ of non-empty subsets, satisfying $k\in I_i$ $\Leftrightarrow$ there are exactly $m_i\geq 0$ distinct $a\in\C^*$ such that for some matrix $U\in\Trg$, we have
\begin{equation}\label{eq:sys:trsfa}
	\ous(f-q_k)_\Gamma(a) = 0.
\end{equation} Note that this decomposition is not unique.

\begin{example}\label{ex:labeling}
Consider the map $f$ constructed in Section~\ref{subsec:Th3-proof}. We have $\mathcal{C}^\circ=\TT$, $|\mathcal{I}_\emptyset| = 2n$ and $\mu_f=2$. Choose $\Gamma$ to be the pair as in Figure~\ref{fig:supp-ex}.
\[
 \left\lbrace \big(\begin{smallmatrix} 0 \\ 0 \end{smallmatrix}\big),\big(\begin{smallmatrix} 1 \\ 1 \end{smallmatrix}\big)\right\rbrace,~\left\lbrace \cup_{i=1}^n(i-1,i+1)\right\rbrace.
\] For any $q\in\mathcal{I}_\emptyset$, we have 
\[
\ous(f-q)_\Gamma = (t-q,b_0+b_1t+\cdots+b_nt^n),
\] where the choice of $(b_0,b_1,\ldots,b_n)$ is so that $b_0+b_1t+\cdots+b_nt^n$ has exactly $n$ simple roots in $\C^*$. Then, we have $\lambda=1$, with $m_1 = 1$. 
\end{example}

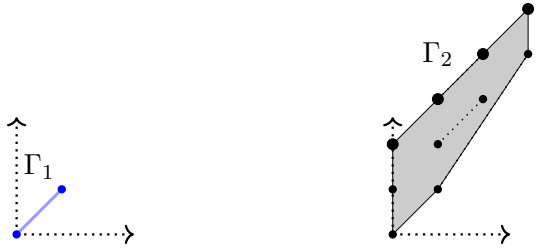
\begin{figure}[h]
\centering
\begin{tikzpicture}
    \tikzstyle{bluefill1} = [fill=blue!20,fill opacity=0.8]          
    \tikzstyle{blackfill1} = [fill=gray,fill opacity=0.4]          
    \tikzstyle{greyfill1} = [fill=gray!20,fill opacity=0.8]          
    \tikzstyle{conefill} = [pattern = north east lines, pattern color=gray]          
    \tikzstyle{ann} = [fill=white,font=\footnotesize,inner sep=1pt] 
    \tikzstyle{ghostfill} = [fill=white]	
    \tikzstyle{ghostdraw} = [draw=black!50]					
	\tikzstyle{ann1} = [font=\footnotesize,inner sep=1pt] 
	\tikzstyle{ann2} = [font=\normalsize,inner sep=1pt] 


\begin{scope}[xshift = -5cm, yshift = 0cm, xscale=1.2, yscale=1.2]


	\draw[arrows=->,line width=0.3 mm, dotted] (0,0)-- (0,1.3); 
	
	\draw[arrows=->,line width=0.3 mm, dotted] (0,0)-- (1.3,0); 

   
   	\draw[ line width=0.4 mm, color = blue!40] (0,0) -- (0.5,0.5); 


		\draw [fill, color = blue] (0,0) circle [radius=0.04];
		\draw [fill, color = blue] (0.5,0.5) circle [radius=0.04];

	\node[ann2] at (0.25,0.75)		{$\Gamma_1$};
 \end{scope}	
	

\begin{scope}[xshift = 0cm, yshift = 0cm, xscale=1.2, yscale=1.2]


	\filldraw[blackfill1,line width=0.0 mm ] (0,0)--(0.5,0.5)--(1.5,2)
										-- (1.5,2.5) -- (0,1)-- (0,0);	


	\draw[arrows=->,line width=0.3 mm, dotted] (0,0)-- (0,1.3); 
	
	\draw[arrows=->,line width=0.3 mm, dotted] (0,0)-- (1.3,0); 



	\draw[ line width=0.2 mm, dotted] (0.5,1.5) -- (1, 2); 	
	\draw[ line width=0.2 mm, dotted] (0.5,1) -- (1, 1.5); 	
	\draw[ line width=0.2 mm, dotted] (0.5,0.5) -- (1.5, 2); 	
		\draw [fill, color = black] (0,0) circle [radius=0.04];
		\draw [fill, color = black] (0,0.5) circle [radius=0.04];
		\draw [fill, color = black] (0,1) circle [radius=0.06];
		\draw [fill, color = black] (0.5,0.5) circle [radius=0.04];
		\draw [fill, color = black] (0.5,1) circle [radius=0.04];
		\draw [fill, color = black] (0.5,1.5) circle [radius=0.06];
		\draw [fill, color = black] (1,1.5) circle [radius=0.04];
		\draw [fill, color = black] (1,2) circle [radius=0.06];
		\draw [fill, color = black] (1.5,2) circle [radius=0.04];
		\draw [fill, color = black] (1.5,2.5) circle [radius=0.06];

	\node[ann2] at (0.5,2)		{$\Gamma_2$};

\end{scope}	
\end{tikzpicture}
\caption{The couple of supports in Example~\ref{ex:labeling}.}\label{fig:supp-ex}
\end{figure}

Let $g_1,g_2\in\C[z_1]$ denote the polynomials $\ous f_{1,\Gamma},\ous f_{2,\Gamma}$ respectively and let $G(r,s)=\underline{0}$ denote the polynomial system 
\begin{equation}\label{eq:sys:nodes1}
  \begin{array}{@{}ccccl@{}}
   g_1(r) - g_1(s) & = & 0,\\
   g_2(r) - g_2(s) & = & 0.   
  \end{array}
\end{equation} Then, Equation~\eqref{eq:sys:trsfa} shows that for any $I_i$ with $m_i\geq 1$, points $a_1,\ldots,a_{m_i}\in\C^*$ above satisfy
\begin{equation}\label{eq:set-nodes}
\{(a_k,a_l)\in\mathbb{V}^\circ(G)~|~k,l=1,\ldots,m_i\}. 
\end{equation} The following observation is obvious.

\begin{claim}\label{claim1}
The set of isolated points in $\mathbb{V}^\circ(G)$ satisfies $r\neq s$.
\end{claim} This shows that~\eqref{eq:set-nodes} identifies a distinguished subset of $
m_i^2 - m_i$ isolated points in $\mathbb{V}^\circ(G)$.

We proceed as follows: Each set $I_i$ distinguishes a unique subset 
\begin{equation}\label{eq:subset1}
\{q_{k}\}_{k\in I_i}
\end{equation} in $\mathcal{I}_\emptyset$. Since the solutions to $\ous(f-q)_\Gamma =0$ vary with $q$ in~\eqref{eq:subset1}, each $q_{k}$ distinguishes a unique set $\{a_1,\ldots,a_{m_i}\}\subset\C^*$. This in turn distinguishes a subset of $m_i(m_i-1)$ isolated points in $\mathbb{V}^\circ(G)$.  Therefore, each $I_i$ gives rise to $m_i(m_i -1)|I_i|$ distinct points in $\mathbb{V}^\circ(G)$. That is, we have 
\begin{equation}\label{eq:miZ}
\sum_{i} m_i(m_i -1)|I_i|\leq \#\mathbb{V}^\circ(G).
\end{equation} In order to provide an upper bound on $N$ above, we consider two cases according to Lemma~\ref{lem:simple-transf-sol}. In what follows, we use $\mu$ as shorthand to $\mu_f$.

\textbf{Assume that $\Gamma$ is the only long relevant face.} Then, Lemma~\ref{lem:simple-transf-sol} shows that $\lambda=1$, with $m_1= \mu $, and Equation~\eqref{eq:miZ} becomes
\begin{equation}\label{eq:mu-ineq}
\mu(\mu - 1)N\leq \#\mathbb{V}^\circ(G).
\end{equation}

\begin{proof}[Proof of Proposition~\ref{prop:S0f}~\eqref{eq:ineqS01}]
 B\'ezout's theorem shows that 
\begin{equation}\label{eq:nod:Bez}
\#\mathbb{V}^\circ(G)\leq \deg g_1\cdot\deg g_2,
\end{equation} and Lemma~\ref{lem:degree-length} shows that for $j=1,2$ we have 
\begin{equation}\label{eq:deg-boundS0}
\deg g_j=\ell(\Gamma_j)\leq \deg f_j/2.
\end{equation} Therefore, we have $N\leq \deg f_1\cdot\deg f_2/(4\mu(\mu-1))$.
\end{proof} 

\begin{proof}[Proof of Proposition~\ref{prop:S0f}~\eqref{eq:ineqS02}]
 Remark~\ref{rem:dom-indep} shows that for some $i\in\{1,2\}$, say $i=1$, we have $A_i\neq\Gamma_i$.

For any $w\in A_1\setminus \Gamma_1$, we have $\conv(\sigma)\subset \conv(A_1)$ and $\conv(\Gamma_2)\subset\conv(A_2)$. Then, the inequality
\begin{equation}\label{eq:mu=V-S0}
V(\sigma,\Gamma_2)\leq V(A)= \mu
\end{equation} is deduced from Lemma~\ref{lem:gen-non-degree} and from the following fact.
\begin{fact}[monotonicity, see~\cite{Sch14}, Chapter 5.25]\label{fact}
If $L_1,L_2,\Delta_1,\Delta_2\subset\R^2$ are convex bodies such that $L_1\subset\Delta_1$ and $L_2\subset\Delta_2$, then $ V(L_1,L_2)\leq V(\Delta_1,\Delta_2)$.
\end{fact} Finally, Lemma~\ref{lem:mixed-segments} shows that 
\begin{equation}\label{eq:mixed-length}
\ell(\Gamma_2)\leq \ell(\sigma)\cdot\ell(\Gamma_2)\leq V(\sigma,\Gamma_1).
\end{equation} Therefore, we conclude 
 \[N\overset{~\eqref{eq:mu-ineq}}{\leq}\frac{\mathbb{V}^\circ (G)}{\mu(\mu - 1)} \overset{~\eqref{eq:nod:Bez},~\eqref{eq:deg-boundS0}}{\leq} \frac{\ell(\Gamma_1)\cdot\ell(\Gamma_1)}{\mu(\mu - 1)} \overset{~\eqref{eq:mu=V-S0},~\eqref{eq:mixed-length}}{\leq} \ell(\Gamma_1)\overset{\text{Lemma}~\ref{lem:degree-length} }{\leq}\frac{d_1}{2}.\] 
\end{proof}
\textbf{Assume that there is another long bad face $\Lambda\prec A$.} Lemma~\ref{lem:simple-transf-sol} shows that for each $I_i$, there exists $V\in\Trsf_\Lambda(\Z^2)$, such that we have the following setup. For each $q\in\{q_k\}_{k\in I_i}$, there exists a set $\mathcal{A}_q$ of $m_i$ values $a\in\C^*$, and a set $\mathcal{B}_q$ of $n_i$ distinct values $b\in\C^*$ such that $m_i+n_i = \mu$ and
\begin{equation}\label{eq:sys:UV}
  \begin{array}{@{}ccccl@{}}
   \ous (f-q)_\Gamma(a) & = & 0,\\
   \overline{V}(f-q)_\Lambda(b) & = & 0.   
  \end{array}
\end{equation}
\begin{example}[\ref{ex:labeling} continued]\label{ex:labeling2}
The face $\Lambda$ here is  $ \left(\left\lbrace \big(\begin{smallmatrix} 0 \\ 0 \end{smallmatrix}\big),\big(\begin{smallmatrix} 1 \\ 1 \end{smallmatrix}\big)\right\rbrace\right),~ \left(\left\lbrace \big(\begin{smallmatrix} 0 \\ 0 \end{smallmatrix}\big),\big(\begin{smallmatrix} 1 \\ 1 \end{smallmatrix}\big)\right\rbrace\right)
$, and $\overline{V}(f-q)$ is the pair $(t-q_1,t-q_2)$. From Example~\ref{ex:labeling}, we have $\lambda=1$, and $m_1=1$. Then, for each $q\in\cup_{t\in\mathbb{V}(P)}\{(t,t)\}$ (see Section~\ref{subsec:Th3-proof}), we have $n_1=1$, $\mathcal{A}_q=\mathcal{B}_q=\{t\}$.
\end{example}
For $j=1,2$, we define $g_j$ and $h_j$ to be $\ous f_{j,\Gamma}$ and $\overline{V}g_{j,\Gamma}$ respectively. Then, the sets
\begin{equation}\label{eq:ABq}
  \begin{array}{@{}ccccccc@{}}
   \mathcal{A}_q\times \mathcal{A}_q, & \ & \mathcal{B}_q\times \mathcal{B}_q, & \ & \mathcal{A}_q\times \mathcal{B}_q,
  \end{array}
\end{equation} form subsets of solutions to the systems $G=0$, $H=0$ and $K=0$ respectively, defined by
\begin{equation}\label{eq:sys:big}
  \begin{array}{@{}ccccccc@{}}
   g_j(s) & - & g_j(t) & = & 0, & j=1,2\\
   h_j(u) & - & h_j(v) & = & 0, & j=1,2\\   
   g_j(s) & - & h_j(u) & = & 0, & j=1,2.
  \end{array}
\end{equation} The following Claim is similar to Claim~\ref{claim1}.
\begin{claim}\label{claim:2}
The set of non-isolated points in $\mathbb{V}^\circ(G)$,$\mathbb{V}^\circ(H)$, or $\mathbb{V}^\circ(K)$ satisfy $s=t$.
\end{claim} This shows that for each above $q$, there exists a unique subset~\eqref{eq:ABq} identifying a subset of 
\[
m_i^2 - m_i + n_i^2 - n_i + m_in_i
\] isolated points in $\mathbb{V}^\circ(G)\cup\mathbb{V}^\circ(H)\cup\mathbb{V}^\circ(K)$. Then, the above discussion shows that
\begin{equation}\label{eq:sum-smaller-Sum}
\sum_{i}(m_i^2 - m_i + n_i^2 - n_i + m_in_i)\cdot |I_i|\leq \#\mathbb{V}^\circ(G)+ \#\mathbb{V}^\circ(H)+\#\mathbb{V}^\circ(K).
\end{equation} We proceed by computing the upper bounds in Proposition~\ref{prop:S0f}.

Using B\'ezout's theorem, on the systems~\eqref{eq:sys:big} and the inequality of Lemma~\ref{lem:degree-length}, we bound the value in~\eqref{eq:sum-smaller-Sum} by 
\begin{equation}\label{eq:sum-smaller-Sum2}
\ell(\Gamma_1)\cdot\ell(\Gamma_2) + \ell(\Lambda_1)\cdot\ell(\Lambda_2) + \max(\ell(\Gamma_1),\ell(\Lambda_1))\cdot\max(\ell(\Gamma_2),\ell(\Lambda_2)). 
\end{equation} Moreover, the equality part in Lemma~\ref{lem:degree-length} produces the upper bound 
\begin{equation}\label{eq:bound34}
3\deg f_1\deg f_2/4
\end{equation} for~\eqref{eq:sum-smaller-Sum}. Next, we use $N=\sum I_i$, $m_i = \mu - n_i$ in the l.h.s of~\eqref{eq:sum-smaller-Sum}, expand it, then divide by $\mu^2$ to obtain
\begin{equation}\label{eq:kmu2}
N + \sum i\big(\frac{n_i^2}{\mu^2} - \frac{n_i}{\mu} - \frac{1}{\mu} \big)\leq \frac{k}{\mu^2},
\end{equation} where $k$ is the minimum of~\eqref{eq:sum-smaller-Sum2} and~\eqref{eq:bound34}.

Identities $n_i/\mu\leq 1$ and $N=\sum I_i$ imply
\[
\sum i\big(\frac{n_i^2}{\mu^2} - \frac{n_i}{\mu} - \frac{1}{\mu} \big)\leq - \frac{N}{\mu}.
\] Therefore,~\eqref{eq:kmu2} yields
\begin{equation}\label{eq:kmu21}
N \leq \frac{k}{\mu(\mu - 1)}.
\end{equation} 

\begin{proof}[Proof of Proposition~\ref{prop:S0f}~\eqref{eq:ineqS01}] It follows from this inequality for $k=$~\eqref{eq:bound34} and $\mu\geq 2$.
\end{proof}

\begin{proof}[Proof of Proposition~\ref{prop:S0f}~\eqref{eq:ineqS02}] Without loss of generality, we suppose that $\dim A_2 = 2$ (this is possible since $(A_1,A_2)$ has two relevant long faces). Moreover, we may also suppose that $\ell(\Gamma_1)\geq \ell(\Lambda_1)$.

Then, there exists a subset $\sigma\in A_2$ such that $\dim (\sigma +\gamma_1) = 2$. We get
\[
\ell(\sigma)\cdot\ell(\Gamma_1)\overset{\text{Lemma}~\ref{lem:mixed-segments}}{\leq} V(\sigma,\Gamma_1)\overset{\text{Fact}~\ref{fact}}{\leq} V(A) \overset{\text{Lemma}~\ref{lem:gen-non-degree}}{=} \mu.
\] Replacing $k$ by~\eqref{eq:sum-smaller-Sum2}, above inequalities together with~\eqref{eq:kmu21} yield 
\begin{multline*}
N\leq\frac{1}{(\mu - 1)\ell(\Gamma_1)\cdot\ell(\sigma)} \big( \ell(\Gamma_1)\cdot\ell(\Gamma_2) + \ell(\Gamma_1)\cdot\ell(\Lambda_2) + \ell(\Gamma_1)\cdot\max(\ell(\Gamma_2),\ell(\Lambda_2)) \big).\\ 
\overset{\text{Lemma}~\ref{lem:degree-length}}{\leq} 3\deg f/2(\mu - 1).
\end{multline*}
 \end{proof}

\subsection{Geometry of the Jenonek set and missing points of maps}\label{subs:geo-interpr} Upper bounds for $\#\mathbb{V}^\circ(G)+\#\mathbb{V}^\circ(H)+\#\mathbb{V}^\circ(K)$ appearing above estimate the number of some nodes of the Jelonek set. An instance of this is when there exists only one long face $\Gamma\prec A$ that happens to be origin. Then, 
\[
\left\lbrace y\in\C^2~|~\ous(f-y)_\Gamma(a)=0~\text{for some $a\in\C^*$}\right\rbrace
\] describes a part of a rational curve 
\[
C:=\left\lbrace (g_1,g_2)(s)\in\C^2~|~s\in\C^*\right\rbrace.
\] Therefore, the set $\mathbb{V}^\circ(G)$ describes all self-intersections (nodes) that $C$ can have in $\TT$. Namely, if a point $y\in C$ also belongs to $\mathcal{I}_\emptyset\cap \mathcal{C}^\circ\setminus \mathcal{K}$, then it is a node of multiplicity $\mu$. This explains the phenomenon in the map~\eqref{eq:map:ex1} of Section~\ref{sec:intro}. 

As for the case where there are two long relevant faces $\Gamma,\Lambda\prec A$, the set $\mathbb{V}^\circ(G)\cup\mathbb{V}^\circ(H)\cup \mathbb{V}^\circ(K)$ describes both the intersection locus $C\cap C'$ (with $C'$ similarly defined as $C$, using $H$) and the nodes of the curves. Here, each of $C$ and $C'$ can be a rational curve, or a union of vertical (horizontal) lines in $\C^2$. In the map~\eqref{eq:map:ex3}, the set $\mathcal{I}^\emptyset_f\cap \mathcal{C}^\circ\setminus \mathcal{K}$ results from intersections of vertical lines with the line $\{y_1=y_2\}$.
\section{Proof of Proposition~\ref{prop:S1f}}\label{sec:proof-prop2} We consider the generically non-proper map $f$ from the previous section. Recall from Lemma~\ref{lem:Jel-discr} that $\mathcal{K}$ coincides with the set of all $y\in\C^2$ at which there exists a face $\Gamma\prec A$ and a matrix $U\in\Trg$ such that $\ous(f-y) = 0$ has a double solution in $\C^*\times\{0\}$. 

In this section, we prove the inequality
\begin{equation}\label{eq:ineq:S1}
|\mathcal{I}_\emptyset\cap\mathcal{C}^+\setminus\mathcal{K} |\leq \deg f_1 + \deg f_2,
\end{equation} where $\mathcal{C}^+:=\mathcal{C}^+_f$ is the cross centered at $f(0,0)$ (also, it is the set $\C^2\setminus\mathcal{C}^\circ$). 

In the notations of Proposition~\ref{prop:main}, for any $y\in\mathcal{I}_\emptyset\setminus\mathcal{K}$, there are $\mu$ points in $\C^*\times\{0\}$ distributed among simple solutions to systems $\ous(f-y) = 0$. 

We deduce from $y\in\C^2\setminus f(\C^2)$  and from Lemma~\ref{lem:solutions-infty2} that all above matrices $U\in\Trg$ correspond to only relevant faces $\Gamma\prec A$.

\begin{claim}\label{claim:3}
For each $y\in\mathcal{I}_\emptyset\setminus\mathcal{K}$, at least one of the above relevant faces of $A$ is long.
\end{claim}

\begin{proof}
Lemma~\ref{lem:simple-transf-sol} takes care of the case where $y\in\mathcal{C}^\circ$. 

Assume that $y\in\mathcal{C}^+$, and let $\Gamma$ be a short pertinent face of $A$. One can check that the set of all $y\in\C^2$ such that $\ous(f-y)_\Gamma = 0$ has a solution in $\C^*$, is a vertical/horizontal line in $\mathcal{C}^+$. 

If all faces $\Gamma$ are short, then there exists $\mu$ solutions in $\C^*$ to systems above determining the same line $L\subset \mathcal{C}^+$. Therefore, $L\subset \C^2\setminus f(\C^2)$, and $y$ is a missing point that is not isolated.
\end{proof} Thanks to the above Claim, it is enough to give an upper bound on the total number of distinct $y\in\mathcal{C}^+$ such that $\ous(f-y)_\Gamma = 0$ has a solution in $\C^*$ for some long relevant face $\Gamma\prec A$.

We assume that $y$ belongs to the horizontal line $H$ of $\mathcal{C}^+$. Then, the second coordinate of $y$ equals to $f_2(0,0)$. To determine the first coordinate, we solve the (now, triangular) square system $\ous(f-y)_\Gamma = 0$, with $y_2 = f_2(0,0)$.

Lemma~\ref{lem:degree-length} applied to $\ous(f_2-f_2(0))_\Gamma$ shows that we obtain at most $\deg f_2/2$ distinct points in $H$ for each system $\ous(f-y)_\Gamma = 0$. Lemma~\ref{lem:le-ri-lo-sh(1)} shows that there are at most two such systems with $\Gamma$ being long relevant. This amounts to the upper bound $\deg f_2$ if $y\in H$. By symmetry on the vertical line of $\mathcal{C}^+$, we obtain the upper bound $\deg f_1+\deg f_2$.
\section{Critical points at infinity and proof of Proposition~\ref{prop:double-points} }\label{sec:proof-prop3} Keeping with the same notations as in the previous section, let $f$ be a generically non-proper map. To prove 
$
|\mathcal{I}_\emptyset\cap\mathcal{K} |\leq \deg f_1 + \deg f_2
$, we introduce two technical results. Write $\mathcal{I}_\emptyset\cap\mathcal{K}$ as
\begin{equation}\label{eq:split-union}
 \bigcup_{\substack{\Gamma \prec A\\ \Gamma~\text{relevant}}} \mathcal{K}(\Gamma),
\end{equation} where $\mathcal{K}(\Gamma)$ consists of all $y\in\mathcal{I}_\emptyset$ such that $\ous(f-y)=0$ has a solution in $\C^*\times\{0\}$ of multiplicity at least two for some matrix $U\in\Trg$.
\begin{lemma}\label{lem:le-ri-lo-sh(2)}
Let $\Gamma$ be a right (left) long origin relevant face of $A$. Then, there are no right (left) relevant faces of $A$ other than $\Gamma$.
\end{lemma}
\begin{proof}
Let $\alpha:=(\alpha_1,\alpha_2)\in\Q_-\times\Q_+$ be a supporting vector of $\Gamma$ assuming it is right. Then, for any $p\in\Q^*$, the vector $\alpha^p:=(\alpha_1,\alpha_2+p)$ supports a face of $A$ different form $\Gamma$. All right faces are determined this way.

If $p>0$, then $\alpha^p$ determines the face $\big(\left\lbrace \big(\begin{smallmatrix} 0\\ 0 \end{smallmatrix}\big)\right\rbrace,~\left\lbrace \big(\begin{smallmatrix} 0\\ 0 \end{smallmatrix}\big)\right\rbrace \big)$. Otherwise, the face determined by $\alpha^p$ is not relevant.
\end{proof}
The following result will be proven at the end of this section. Recall that for any finite set $\sigma\subset\Z^2$, we define $\ell(\sigma):=|\conv(\sigma)\cap\Z^2|-1$.

\begin{proposition}\label{prop:mult}
Let $\Gamma:=(\Gamma_1,\Gamma_2)$ be a pertinent face of $A$. Then, we have

\begin{enumerate}[label=(\Alph*)]

	\item\label{it:mult:deg-deg} $|\mathcal{K}(\Gamma)|\leq (\deg f_1 + \deg f_2)/2$,
	
	\item\label{it:mult:len-deg} $|\mathcal{K}(\Gamma)|\leq \ell(\Gamma_i) + \deg f_j/2$ if $\Gamma_j$ is a vertex of $A_j$ for any distinct $i,j\in\{1,2\}$ and
	
	\item\label{it:mult:len-len}  $|\mathcal{K}(\Gamma)|\leq \ell (\Gamma_i)$ if $\Gamma_i$ does not contain $(0,0)$ for some $i\in\{1,2\}$.
\end{enumerate}
\end{proposition}

\begin{proof}[Proof of Proposition~\ref{prop:double-points}] Let us split the union~\eqref{eq:split-union} into two disjoint subsets  
\[
M^l\sqcup M^r,
\] formed by the contribution of left faces of $A$, and the other by its right ones. The result will follow by showing that each of $|M^r|$ and $|M^l|$ is bounded by $(\deg f_1+\deg f_2)/2$. By symmetry, it suffices to give an upper bound for $|M^l|$.

Assume first that $A$ has a long right relevant origin face (recall Definition~\ref{def:face-non-prop}). From Lemma~\ref{lem:le-ri-lo-sh(2)}, it is the unique right relevant face. Item~\ref{it:mult:deg-deg} of Proposition~\ref{prop:mult} shows that 
\[
|M^l|\leq (\deg f_1+\deg f_2)/2.
\]

Assume now that $A$ does \emph{not} have long origin right relevant faces (In Figure~\ref{fig:couple-short}, for example, the only long relevant face is not origin). 

Denote by $\Gamma(0)$ the origin right short face of $A$, and suppose that $\Gamma_1(0)\subset A_1$ is the one-dimensional member of $\Gamma(0)$ (we have $\Gamma(0)=(a,\blacksquare)$ in  Figure~\ref{fig:couple-short}). Then, Item~\ref{it:mult:len-deg} of Proposition~\ref{prop:mult} shows that 
\[
|\mathcal{K}(\Gamma (0))|\leq\ell (\Gamma_1(0)) + \deg f_2/2.
\] We have the following easy fact: 

\emph{Each face $\Gamma(i)\prec A$ in the collection of all other bad left faces $\{\Gamma(1),\ldots,\Gamma(k)\}$ is half-origin, and its first member $\Gamma_1(i)\subset A_1$ does not contain $(0,0)$.} 

In Figure~\ref{fig:couple-short}, for example, we have $k=3$, $\Gamma(1)=(b,\blacksquare)$, $\Gamma(2)=(c,\blacksquare)$ and $\Gamma(3)=(d,e)$. Therefore, Item~\ref{it:mult:len-len} of Proposition~\ref{prop:mult} shows that 
\[
\sum_{i=1}^{k}|\mathcal{K}(\Gamma(i))|\leq \sum_{i=1}^k\ell(\Gamma_1(i)).
\] Finally, note that the collection $\Gamma_1(0),\Gamma_1(1),\ldots,\Gamma_1(k)$ forms a chain of left faces of $A_1$. Therefore, the result follows from $\sum_{i=0}^k\ell(\Gamma_1(i))\leq \deg f_1/2$.
\end{proof}

\subsection{Proof of Proposition~\ref{prop:double-points}} Pick any matrix $U\in\Trg$. The following Lemma shows that it suffices to prove the result only for $U$.

\begin{lemma}\label{lem:multip}
The set $\mathcal{K}(\Gamma)$ does not depend on the choice of the matrix in $\Trg$.
\end{lemma}

\begin{proof}
The unimodularity of matrices in $\Trg$ implies that the number of solutions to systems, as well as their multiplicities, is preserved. Moreover, all transformed systems under $U$ are equal when restricted to $\Gamma$. Therefore, any transformation of $f-y $ under $U$ gives a unique number of solutions in $\C^*\times\{0\}$, together with multiplicities.
\end{proof}

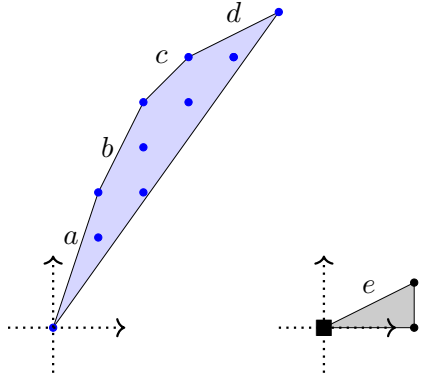
\begin{figure}
\centering
\begin{tikzpicture}
    \tikzstyle{conefill1} = [fill=blue!20,fill opacity=0.8]          
    \tikzstyle{conefill} = [pattern = north east lines, pattern color=gray]          
    \tikzstyle{ann} = [fill=white,font=\footnotesize,inner sep=1pt] 
    \tikzstyle{ghostfill} = [fill=white]	
    \tikzstyle{ghostdraw} = [draw=black!50]					
	\tikzstyle{ann1} = [font=\footnotesize,inner sep=1pt] 
	\tikzstyle{bluefill1} = [fill=blue!20,fill opacity=0.8]          
    \tikzstyle{blackfill1} = [fill=gray,fill opacity=0.4]          
    \tikzstyle{greyfill1} = [fill=gray!20,fill opacity=0.8]          
	\tikzstyle{ann1} = [font=\footnotesize,inner sep=1pt] 
	\tikzstyle{ann2} = [font=\normalsize,inner sep=1pt] 

\begin{scope}[xshift = 3cm, yshift = 0.5cm, xscale=1.2, yscale=1.2]


	\filldraw[bluefill1,line width=0.0 mm ](0,0)--(0.5,1.5)--(1,2.5)
										-- (1.5,3) -- (2.5,3.5)-- (0,0) ;	

	\draw [fill, color = blue ] (2,3) circle [radius=0.04];

	\draw [fill, color = blue ] (2.5,3.5) circle [radius=0.04];

	\draw [fill, color = blue] (0,0) circle [radius=0.04];	
	
	\draw [fill, color = blue] (0.5,1) circle [radius=0.04];

	\draw [fill, color = blue] (1,2) circle [radius=0.04];
		
	\draw [fill, color = blue] (1,2.5) circle [radius=0.04];

	\draw [fill, color = blue] (0.5,1.5) circle [radius=0.04];

	\draw [fill, color = blue] (1,1.5) circle [radius=0.04];

	\draw [fill, color = blue] (1.5,2.5) circle [radius=0.04];
		
	\draw [fill, color = blue] (1.5,3) circle [radius=0.04];

	\draw [fill, color = blue] (2,3) circle [radius=0.04];


	\draw[arrows=->,line width=0.3 mm, dotted] (0,-0.5)-- (0,0.8); 
	
	\draw[arrows=->,line width=0.3 mm, dotted] (-0.5,0)-- (0.8,0); 


		\node[ann2] at (0.2,1)   {$a$};	
		\node[ann2] at (0.6,2)   {$b$};	
		\node[ann2] at (1.2,3)   {$c$};
		\node[ann2] at (2,3.5)   {$d$};

\begin{scope}[xshift = 3 cm, yshift = 0cm]	


	\filldraw[blackfill1,line width=0.0 mm ](0,0)--(1,0.5)--(1,0)-- (0,0);	


%
%
	

	\draw [fill, color = black ] (1,0.5) circle [radius=0.04];

	\draw [fill, color = black ] (1,0) circle [radius=0.04];
	

	\draw[arrows=->,line width=0.3 mm, dotted] (0,-0.5)-- (0,0.8); 
	
	\draw[arrows=->,line width=0.3 mm, dotted] (-0.5,0)-- (0.8,0);

		\node[ann1] at (0,0)   {$\blacksquare$};
		\node[ann2] at (0.5,0.45)   {$e$};

\end{scope}

\end{scope}

\end{tikzpicture}
\caption{A pair having several short left bad faces}\label{fig:couple-short}
\end{figure}

Replace $(z_1,z_2)$ by $(s,t)$ and $\ous (f-y)$ by $g:=(g_1,g_2)$, that is $g_1,g_2\in\C[s,t,y_1,y_2]$. 

Assume first that $\Gamma$ is an origin face of $A$. Then, for $i=1,2$, the polynomial $g_i$ is expressed as
\[
\sum_{j=0}^{m_i}t^j g_{i,j}(s) - y_i.
\] A point $y\in\mathcal{K}$ satisfies
\begin{equation}\label{eq:sys:critical}
  \begin{array}{@{}cccccccc@{}}
   g_{1,0}  -  y_1 & =  & g_{2,0}  -  y_2 & = & D & = & 0,
  \end{array} 
\end{equation} where $D\in\C[s]$  is the polynomial determining $|\Jac g|$ at $t =0$.

On the one hand,~\eqref{eq:sys:critical} has at most $\deg D$ distinct solutions in $\C^3$. This implies that 
\[
|\mathcal{K}(\Gamma)|\leq \deg D.
\] On the other hand, we have 
\[
D=g_{1,1}\cdot g_{2,0}' - g_{2,1}\cdot g'_{1,0},
\] where the derivation is w.r.t. $z_1$. Therefore, Item~\ref{it:mult:deg-deg} of Proposition~\ref{prop:mult} follows form Lemma~\ref{lem:degree-length} by observing that $g_{i,j} = g_{i|\sigma}$, for some horizontal set of points $\sigma\subset\N^2$. Item~\ref{it:mult:len-deg} of Proposition~\ref{prop:mult} follows similarly by additionally noting that $\deg g_{i,0}=\ell(\Gamma_i)$.

Assume that one of the members, say, $\Gamma_2$, of $\Gamma$ does not contain $(0,0)$. Then, the polynomial $g_2$ is written as 
\[
\sum_{j=0}^{m_i}t^jg_{2,j}(s) - s^vt^wy_2,
\] where $(v,w)\in\N^2$. The polynomial $\tilde{D}$ determining $|\Jac g|$ at $t =0$ is a polynomial in $s$ and $y_2$
\[
(g_{2,1} - sy_2\delta_{k1})\cdot g_{1,0}' - g_{1,1}\cdot g'_{2,0},
\] where $\delta_{k1}$ is the Kronecker delta. The system~\eqref{eq:sys:critical} is written as $ g_{1,0}  -  y_1  =  g_{2,0}  =  \tilde{D} = 0$. If $k=1$, we get 
\[
|\mathcal{K}(\Gamma)|\leq \deg g_{2,0}\leq \ell (\Gamma_2),
\] from using similar arguments as in the previous cases ($\Gamma$ origin). This proves~\ref{it:mult:len-len} ($\Rightarrow$~\ref{it:mult:len-deg}) for $k=1$, and thus Lemma~\ref{lem:degree-length} yields~\ref{it:mult:deg-deg}.

We finish the proof by showing that $f$ is generically non-proper $\Rightarrow k=1$. 

Assume that $k\neq 1$. Then, $\tilde{D}$ depends only on $s$, and the system $ g_2  =  0, |\Jac g| =0$ has a solution in $\C^*\times\{0\}$. One can check, using the chain rule for partial derivations that the sets  $\{|\Jac g|=0\}$ and  $\{\ous |\Jac f|)=0\}$  are equal in $\TT$. Then, Proposition~\ref{prop:main} shows that \[n_0:=\#\mathbb{V}^\circ (g_2,\ous |\Jac f|) < V(\supp g_2, \supp \ous J\!f).
\] Since elements in $\Trg$ preserve multiplicities for solutions to the transformed systems, we deduce from Lemma~\ref{lem:multip} that 
\[
\#\mathbb{V}^\circ(f_2, |\Jac f|) = n_0.
\] Finally, since $V(\supp g_2, \supp \ous |\Jac f|) = V(\supp f_2,\supp |\Jac f|)$~\cite{Ber75}, we get 
\[
n_0<V(\supp f_2,\supp \Jac f),
\] which implies that $f$ is not generically non-proper, a contradiction. 
\subsection{Cusps of the Jelonek set}\label{subs:geo-interp2} In contrast to Section~\ref{subs:geo-interpr}, singularities of the Jelonek set $\mathcal{J}_f$ that are not nodes or self-intersection, can be also good candidates for isolated missing points of $f$.

Consider the map~\eqref{eq:map:ex2} defined in Section~\ref{sec:intro}, whose Jelonek set has a cusp coinciding with the unique missing point $(0,1)$. This is not a coincidence, and the reason for this goes as follows.

On the one hand, solutions in $\C^*\times\C^2$ to the system~\eqref{eq:sys:critical} account for all cusps that a parametrized curve $C$ (from Section~\ref{subs:geo-interpr}) may have. On the other hand, Lemma~\ref{lem:solutions-infty2}, applied to the unique relevant face $\Gamma$ of $A$, where 
\[
\Gamma_1 = \left\lbrace\big(\begin{smallmatrix} 0\\ 0 \end{smallmatrix}\big),\big(\begin{smallmatrix} 1\\ 1 \end{smallmatrix}\big),\big(\begin{smallmatrix} 2\\ 2 \end{smallmatrix}\big)\right\rbrace,~\Gamma_2 = \left\lbrace\big(\begin{smallmatrix} 0\\ 0 \end{smallmatrix}\big),\big(\begin{smallmatrix} 1\\ 1 \end{smallmatrix}\big),\big(\begin{smallmatrix} 2\\ 2 \end{smallmatrix}\big),\big(\begin{smallmatrix} 3\\ 3 \end{smallmatrix}\big)\right\rbrace,
\] shows that $\mathcal{J}_f$ is expressed as the set of all $y\in\C^2$ at which $y=\ous f_\Gamma (s)$ for some matrix $U\in\Trg$. This makes $\mathcal{J}_f$ a curve whose $i$-th coordinate has the above parametrization.
\subsection*{Contact}
  Boulos El Hilany,
  Johann Radon Institute for Computational and Applied Mathematics,
 Altenberger Stra\ss e 69,
 4040 Linz,
 Austria;
\href{mailto:boulos.hilani@gmail.com}{boulos.hilani@gmail.com}, \href{https://boulos-elhilany.com}{boulos-elhilany.com}.

\bibliographystyle{abbrv}					   

\bibliography{C:/Users/boulo/Dropbox/Personal_data/Maths/Latex-Paths/Bibliography/mainbib}

\def\cprime{$'$}
\begin{thebibliography}{10}

\bibitem{Ber75}
D.~N. Bernstein.
\newblock The number of roots of a system of equations.
\newblock {\em Funkcional. Anal. i Prilo\v zen.}, 9(3):1--4, 1975.

\bibitem{CCDDS13}
E.~Cattani, M.~A. Cueto, A.~Dickenstein, S.~Di~Rocco, and B.~Sturmfels.
\newblock Mixed discriminants.
\newblock {\em Math. Z.}, 274(3-4):761--778, 2013.

\bibitem{DSS09}
M.~Drton, B.~Sturmfels, and S.~Sullivant.
\newblock {\em Lectures on algebraic statistics}, volume~39 of {\em Oberwolfach
  Seminars}.
\newblock Birkh\"{a}user Verlag, Basel, 2009.

\bibitem{EH19desc}
B.~El~Hilany.
\newblock Describing the {J}elonek set of polynomial maps via {N}ewton
  polytopes.
\newblock {\em arXiv preprint arXiv:1909.07016}, 2019.

\bibitem{EH19anote}
B.~El~Hilany.
\newblock A note on generic polynomial maps having a fiber of maximal
  dimension.
\newblock {\em arXiv preprint arXiv:1910.01333, to appear in Colloquium
  Mathematicum}, 2021.

\bibitem{Est10}
A.~Esterov.
\newblock {N}ewton polyhedra of discriminants of projections.
\newblock {\em Discrete Comput. Geom.}, 44(1):96--148, 2010.

\bibitem{Est13}
A.~Esterov.
\newblock The discriminant of a system of equations.
\newblock {\em Adv. Math.}, 245:534--572, 2013.

\bibitem{Fer14}
J.~F. Fernando.
\newblock On the one dimensional polynomial and regular images of {${\Bbb
  R}^n$}.
\newblock {\em J. Pure Appl. Algebra}, 218(9):1745--1753, 2014.

\bibitem{Fer16}
J.~F. Fernando.
\newblock On the size of the fibers of spectral maps induced by semialgebraic
  embeddings.
\newblock {\em Math. Nachr.}, 289(14-15):1760--1791, 2016.

\bibitem{Fer03}
J.~F. Fernando and J.~Gamboa.
\newblock Polynomial images of {${\Bbb R}^n$}.
\newblock {\em J. Pure Appl. Algebra}, 179(3):241--254, 2003.

\bibitem{FerGam11}
J.~F. Fernando, J.~Gamboa, and C.~Ueno.
\newblock On convex polyhedra as regular images of {${\Bbb R}^n$}.
\newblock {\em Proc. Lond. Math. Soc.}, 103(5):847--878, 2011.

\bibitem{FerUeno14}
J.~F. Fernando and C.~Ueno.
\newblock On complements of convex polyhedra as polynomial and regular images
  of {${\Bbb R}^n$}.
\newblock {\em Int. Math. Res. Not. IMRN}, 2014(18):5084--5123, 2014.

\bibitem{GKZ94}
I.~M. Gel'fand, M.~M. Kapranov, and A.~V. Zelevinsky.
\newblock {\em Discriminants, resultants, and multidimensional determinants}.
\newblock Mathematics: Theory \& Applications. Birkh\"{a}user Boston, Inc.,
  Boston, MA, 1994.

\bibitem{ZiRi03}
R.~Hartley and A.~Zisserman.
\newblock {\em Multiple View Geometry in Computer Vision}.
\newblock Cambridge University Press, 2003.

\bibitem{Jel93}
Z.~Jelonek.
\newblock The set of points at which a polynomial map is not proper.
\newblock In {\em Annales Polonici Mathematici}, volume~58, pages 259--266,
  1993.

\bibitem{Jel99a}
Z.~Jelonek.
\newblock A number of points in the set {${\Bbb C}^2\backslash F ({\Bbb
  C}^2)$}.
\newblock {\em Bull. Pol. Acad. Sci. Math}, 47(3):257--262, 1999.

\bibitem{Jel99}
Z.~Jelonek.
\newblock Testing sets for properness of polynomial mappings.
\newblock {\em Math. Ann.}, 315(1):1--35, 1999.

\bibitem{Jel01c}
Z.~Jelonek.
\newblock Note about the set {$S_f$} for a polynomial mapping {$f\colon{\Bbb
  C}^2\to{\Bbb C}^2$}.
\newblock {\em Bull. Polish Acad. Sci. Math.}, 49(1):67--72, 2001.

\bibitem{Jel02}
Z.~Jelonek.
\newblock Geometry of real polynomial mappings.
\newblock {\em Mathematische Zeitschrift}, 239(2):321--333, 2002.

\bibitem{Lam70}
G.~Laman.
\newblock On graphs and rigidity of plane skeletal structures.
\newblock {\em J. Eng. Math}, 4(4):331--340, 1970.

\bibitem{Las15}
J.~B. Lasserre.
\newblock {\em An introduction to polynomial and semi-algebraic optimization}.
\newblock Cambridge Texts in Applied Mathematics. Cambridge University Press,
  Cambridge, 2015.

\bibitem{NZ90}
A.~N{\'e}methi and A.~Zaharia.
\newblock On the bifurcation set of a polynomial function and {N}ewton
  boundary.
\newblock {\em Publ. Res. Inst. Math. Sci.}, 26(4):681--689, 1990.

\bibitem{Sch14}
R.~Schneider, G.~Rota, B.~Doran, P.~Flajolet, M.~Ismail, T.~Lam, and E.~Lutwak.
\newblock {\em Convex Bodies: The Brunn-Minkowski Theory}.
\newblock Encyclopedia of Mathematics and its Applications. Cambridge
  University Press, 1993.

\bibitem{St94}
B.~Sturmfels.
\newblock On the {N}ewton polytope of the resultant.
\newblock {\em J. Algebr. Comb.}, 3(2):207--236, 1994.

\bibitem{Sta07}
A.~Valette-Stasica.
\newblock Asymptotic values of polynomial mappings of the real plane.
\newblock {\em Topology Appl.}, 154(2):443--448, 2007.

\bibitem{Zah96}
A.~Zaharia.
\newblock On the bifurcation set of a polynomial function and {N}ewton
  boundary, {II}.
\newblock {\em Kodai Math. J.}, 19(2):218--233, 1996.

\end{thebibliography}

\end{document}